\documentclass[a4paper, 11pt, twoside, leqno, openright, slashbox]{amsart}

\usepackage[top=0.9in, bottom=0.9in, left=1in, right=1in]{geometry}

\usepackage{amssymb}
\usepackage{amsthm}
\usepackage{amscd}
\usepackage{mathrsfs}
\usepackage{amsopn}
\usepackage{manfnt}

\usepackage{latexsym}
\usepackage{amsmath}
\usepackage{epic,eepic}
\usepackage{pifont}
\usepackage{lmodern}
\usepackage{array}	
\usepackage{lscape}

\usepackage{graphicx, color}
\usepackage{tikz}
\usepackage{tikz-cd}
\usepackage{pgfplots}
\usepackage {diagbox}

\usetikzlibrary{intersections}
\usetikzlibrary{positioning}
\usetikzlibrary{arrows}
\usetikzlibrary{patterns}
\usetikzlibrary{decorations.pathreplacing}

\usepackage{url}
\usepackage[all]{xy}
\usepackage{paralist}

\usepackage{shuffle}
\usepackage{here}

\usepackage{mathtools}

\definecolor{e-mail}{rgb}{0,.40,.80}
\definecolor{reference}{rgb}{.20,.60,.22}
\definecolor{mrnumber}{rgb}{.80,.40,0}
\definecolor{citation}{rgb}{0,.40,.80}
\usepackage[breaklinks=true,colorlinks=true,linkcolor=reference, citecolor=citation, urlcolor=e-mail]{hyperref}

\def\bP{{\mathbb P}}

\def\bQ{{\mathbb Q}}

\def\bC{{\mathbb C}}

\def\op{{\mathrm{op}}}

\def\log{{\mathrm{log}}}

\def\ellet{{\ell\text{-\'et}}}

\newtheorem{dfn}{Definition}[section]

\newtheorem{thm}[dfn]{Theorem}

\newtheorem{prop}[dfn]{Proposition}

\newtheorem{cor}[dfn]{Corollary}
\newtheorem{rem}[dfn]{Remark}
\newtheorem{ex}[dfn]{Example}

\theoremstyle{remark}
\theoremstyle{definition}

\newtheorem*{theorem*}{Theorem}

\newtheorem*{definition*}{Definition}

\newtheorem*{Remark*}{Remark}

\newcommand{\depth}{\operatorname{dp}}

\usepackage{fancyhdr} 
\makeatletter

\@addtoreset{equation}{section}
\makeatother

\makeatletter
\newcommand*\snakerightarrow[2][]{%
\mathrel{%
\settowidth\@tempdima{\footnotesize#1}%
\settowidth\@tempdimb{\footnotesize#2}%
\ifdim\@tempdimb>\@tempdima\@tempdima=\@tempdimb\fi
\advance\@tempdima15pt
\tikz[font=\footnotesize,baseline=-2pt]
\draw[decorate,decoration={snake, pre length=3pt, post length=3pt, segment length=6pt, amplitude=2pt},->]
(0,0) -- node[above] {#2} node[below] {#1} (\@tempdima,0);%
}%
}

\makeatother

\begin{document}

\title{
On the 
Landen formula\\for multiple polylogarithms and its\\$\ell$-adic Galois analogue}

\author[D.~Shiraishi]{Densuke Shiraishi}
\dedicatory{In celebration of Professor Hiroaki Nakamura's 60th birthday}
\address{Department of General Education,
National Institute of Technology, Kagawa College,
355 Chokushi-cho, Takamatsu, Kagawa 761-8058 Japan
}
\email{densuke.shiraishi@gmail.com,
shiraishi-d@t.kagawa-nct.ac.jp}
\date{}
\subjclass[2020]{11G55; 11F80, 11R32, 14H30}
\keywords{Landen formula, Complex multiple polylogarithm, $\ell$-adic Galois multiple polylogarithm, Baker-Campbell-Hausdorff sum, Goldberg polynomial}
\thanks{This paper will appear in Low Dimensional Topology, Number Theory and Arithmetic Galois Theory, edited by Pierre Dèbes, Masanori Morishita, Takao Satoh, Akio Tamagawa, to be published by World Scientific.}

\maketitle


\begin{abstract}
In the present paper,
we provide an algebraic and geometric proof of the Landen formula for complex multiple polylogarithms originally established by Okuda and Ueno.
Our approach employs a chain rule of complex KZ solutions arising from the symmetry $z \mapsto \frac{z}{z-1}$ of $\mathbb{P}^1 \backslash \{0,1,\infty\}$.
Furthermore,
by replacing complex KZ solutions
with $\ell$-adic Galois 1-cocycles in this proof,
we obtain the Landen formula for $\ell$-adic Galois multiple polylogarithms.
This formula involves lower weight terms
specific to the $\ell$-adic Galois setting, which originate from 
the higher-order terms of 
the Baker-Campbell-Hausdorff sum ${\rm log}({\rm exp}(-e_1){\rm exp}(-e_0))$.
These lower weight terms are
explicitly 
described by an integral involving Goldberg polynomials.
\end{abstract}

\section{Introduction and main results}

Okuda and Ueno established the following formula.

\begin{thm}[Landen formula for multiple polylogarithms {\cite[Proposition 9]{OU04}}] \label{thm1}
Let $\mathbf{k} \in \bigcup_{d=1}^{\infty} \mathbb{N}^d$.
For the multiple polylogarithm defined by
\[
Li_{\mathbf{k}}(z):=\sum_{0<n_1<\cdots<n_d}\frac{z^{n_d}}{n_1^{k_1}\cdots n_d^{k_d}} \quad (|z|<1),
\]
the following formula holds:
\[
Li_{\mathbf{k}}\left(\dfrac{z}{z-1}\right)+(-1)^{1+\depth(\mathbf{k})} \sum_{\mathbf{J} \succeq \mathbf{k}} Li_{\mathbf{J}}\left(z\right) = 0.
\]
Here, for $\mathbf{k}=(k_1,\dots,k_d)$, we set $\depth(\mathbf{k}):=d$.
The notation $\mathbf{J} \succeq \mathbf{k}$ means that $\mathbf{J}$ is a refinement of $\mathbf{k}$ (see Definition \ref{saibun} for the precise definition).
\end{thm}

\noindent
In \cite[Proposition 9]{OU04},
this formula was derived analytically using elementary calculus.
The primary goal of this paper is to provide a purely algebraic and geometric proof of the formula, which is indispensable for establishing its $\ell$-adic Galois analogue.

Throughout this paper,
for any $\mathbf{k} \in \bigcup_{d=1}^{\infty} \mathbb{N}^d$ and any point $z \in \mathbb{P}^1(\mathbb{C}) \backslash \{0,1,\infty\}$, we regard the complex multiple polylogarithm $Li_{\mathbf{k}}(z)$ as
\begin{align}\label{cpoly1}
Li_{\mathbf{k}}(z): \pi_1^{\rm top}\left(\mathbb{P}^1(\mathbb{C}) \backslash \{0,1,\infty\}; \overrightarrow{01}, z\right) \to \bC,
\quad \gamma \mapsto Li_{\mathbf{k}}(z;\gamma).
\end{align}
Here, $Li_{\mathbf{k}}(z;\gamma)$ denotes a specific complex iterated integral along a topological path $\gamma$ from the tangential base point $\overrightarrow{01}$ to $z$ on $\mathbb{P}^1(\mathbb{C}) \backslash \{0,1,\infty\}$.
See subsection \ref{cmpoly} for the detailed definition of the complex number $Li_{\mathbf{k}}(z;\gamma)$.
In this framework, the above theorem can be refined as follows.

\begin{thm}[The Landen formula for complex multiple polylogarithms] \label{M1}
Let $z$ be a point of $\mathbb{P}^1(\mathbb{C}) \backslash \{0,1,\infty\}$.
Given a path $\gamma \in \pi_1^{\rm top}\left(\mathbb{P}^1(\mathbb{C}) \backslash \{0,1,\infty\}; \overrightarrow{01}, z\right)$, define a path $\gamma'$ associated with $\gamma$ by
\begin{align}\label{gamma'}
\gamma':=\delta \cdot \phi(\gamma) \in \pi_1^{\rm top}\left(\mathbb{P}^1(\mathbb{C}) \backslash \{0,1,\infty\}; \overrightarrow{01}, \dfrac{z}{z-1}\right),
\end{align}
where $\delta \in \pi_1^{\rm top}\left(\mathbb{P}^1(\mathbb{C}) \backslash \{0,1,\infty\}; \overrightarrow{01}, \overrightarrow{0\infty}\right)$ is the path from $\overrightarrow{01}$ to $\overrightarrow{0\infty}$ through the upper half plane, $\phi \in {\rm Aut}\left(\mathbb{P}^1(\mathbb{C}) \backslash \{0,1,\infty\}\right)$ is given by
\[
\phi(z)=\dfrac{z}{z-1},
\]
and paths are composed from left to right.
For any $\mathbf{k} \in \bigcup_{d=1}^{\infty} \mathbb{N}^d$, the following holds:
\begin{align} \label{Main1}
Li_{\mathbf{k}}\left(\dfrac{z}{z-1};\gamma'\right)+(-1)^{1+\depth(\mathbf{k})} \sum_{\mathbf{J} \succeq \mathbf{k}} Li_{\mathbf{J}}\left(z;\gamma\right) = 0.
\end{align}
\end{thm}

We derive this formula from a chain rule of complex KZ solutions
\begin{align} \label{crint1}
G_{\overrightarrow{01}}^{\frac{z}{z-1}, \gamma'}\left(e_0^{\rm dR},e_1^{\rm dR}\right)=G_{\overrightarrow{01}}^{z, \gamma}\left(e_0^{\rm dR},e_\infty^{\rm dR}\right)\cdot {\bf exp}\left(\pi \sqrt{-1} e_0^{\rm dR}\right)
\end{align}
arising from the path composition (\ref{gamma'}).
See subsection \ref{cmpoly} for the definition of the complex KZ solution $G_{\overrightarrow{01}}^{z, \gamma}\left(e_0^{\rm dR},e_1^{\rm dR}\right) \in \mathbb{C} \langle \langle e_0^{\rm dR}, e_1^{\rm dR} \rangle \rangle^{\times}$
and section \ref{pfsec} for the details of the proof.

\begin{rem} \label{rem:z=2}
The case $z=2$ is of particular interest because it is a fixed point of the symmetry $\phi: z \mapsto \frac{z}{z-1}$.
In this instance, while the endpoint $z$ coincides with $\frac{z}{z-1}$, the path $\gamma$ is not homotopic to $\gamma'$.
Thus, the main formula \eqref{Main1} specialized to $z=2$ yields a non-trivial relation among the special values $Li_{\mathbf{k}}\left(2;\gamma' \right)$ and $\{Li_{\mathbf{J}}\left(2;\gamma \right)\}_{\mathbf{J} \succeq \mathbf{k}}$, reflecting the monodromy of (\ref{cpoly1}).
\end{rem}

\begin{center}
\begin{figure}[hbtp]
\caption{Topological paths on ${\mathbb P}^1({\mathbb C})\backslash \{0,1,\infty\}$}
\label{paths2}
\begin{tikzpicture} \label{picture}
\draw (2.1,0) to [out=0,in=180] (5,1.9);

\draw (3.67,1.1) -- (3.48,1.1);
\draw (3.67,1.1) -- (3.67,0.9);

\draw (1,1) -- (1.1,1.1);
\draw (1,1) -- (0.9,1.1);

\draw (0.4,-1.5) -- (0.5,-1.4);
\draw (0.4,-1.5) -- (0.3,-1.4);

\draw (1.9,0) to [out=180,in=90] (0.4,-1.5);
\draw (2.7,-3.4) to [out=180,in=-90] (0.4,-1.5);

\draw (1.9,0) -- (2.1,0);
\draw (2,0) to [out=0,in=-90] (3,1);
\draw (2,2) to [out=0,in=90] (3,1);
\draw (2,2) to [out=180,in=90] (1,1);
\draw (2,0) to [out=180,in=-90] (1,1);

\draw[dotted](-2.5,0)--(6.5,0);

\node at (0.6,2.5) {~};

\node at (0.6,1.1) {$\delta$};

\node at (3.6,1.6) {$\gamma$};
\node at (-0.3,-1.4) {$\phi(\gamma)$};
\node at (2,-0.4) {$0$};
\node at (5.8,-0.4) {$1$};
\node at (-1.8,-0.4) {$\infty$};

\node at (5,1.9) {${\bullet}$};
\node at (5,2.2) {${z}$};
\node at (2.7,-3.4) {${\bullet}$};
\node at (3.6,-3.0) {$\dfrac{z}{z-1}$};

\fill[white](2,0)circle(0.07);
\fill[white](5.6,0)circle(0.07);
\fill[white](-1.6,0)circle(0.07);

\node at (2,0) {${\circ}$};
\node at (5.6,0) {${\circ}$};
\node at (-1.6,0) {${\circ}$};

\end{tikzpicture}
\end{figure}
\end{center}

We then turn to the $\ell$-adic Galois analogue of the above formula for a prime number $\ell$. 
Let $K$ be a subfield of $\mathbb{C}$, and let
\[
G_K:={\rm Gal}(\overline{K}/K)
\]
be the absolute Galois group of $K$.
Let $z$ be a $K$-rational base point of $\mathbb{P}^1 \backslash \{0,1,\infty\}$.
Then, $G_K$ acts naturally on the pro-$\ell$ set $\pi_1^\ellet \left(\mathbb{P}^1_{\overline{K}} \backslash \{0,1,\infty\};\overrightarrow{01},{z}\right)$ of pro-$\ell$ \'etale paths from the tangential base point $\overrightarrow{01}$ to $z$ on $\mathbb{P}^1_{\overline{K}} \backslash \{0,1,\infty\}$.
In order to describe this Galois action explicitly,
Wojtkowiak introduced a family of $\ell$-adic special functions $\{Li^{\ell}_{\mathbf{k}}(z)\}_{\mathbf{k}}$:
\begin{align}\label{lpoly1-intro}
Li^{\ell}_{\mathbf{k}}(z): \pi_1^\ellet \left(\mathbb{P}^1_{\overline{K}} \backslash \{0,1,\infty\};\overrightarrow{01},{z}\right) \times G_K \to \mathbb{Q}_{\ell},~(\gamma,\sigma) \mapsto Li^{\ell}_{\mathbf{k}}(z;\gamma,\sigma).
\end{align}
See subsection \ref{lmpoly} for the detailed definition of the $\ell$-adic number $Li^{\ell}_{\mathbf{k}}(z;\gamma,\sigma)$.
By reinterpreting the proof of the complex case (\ref{Main1}) — replacing the complex KZ solution $G_{\overrightarrow{01}}^{z, \gamma}\left(e_0^{\rm dR},e_1^{\rm dR}\right) \in \mathbb{C} \langle \langle e_0^{\rm dR}, e_1^{\rm dR} \rangle \rangle^{\times}$ with the $\ell$-adic Galois 1-cocycle ${\mathfrak f}^{z,\gamma} \in Z^1\left(G_K, \mathbb{Q}_{\ell} \langle \langle e_0^{\mathrm{B}},
e_1^{\mathrm{B}} \rangle \rangle^{\times}\right)$ arising from the Galois action — we derive the $\ell$-adic Landen formula (\ref{Main2}) from a chain rule
\begin{align} \label{crint2}
{\mathfrak f}_{\sigma}^{\frac{z}{z-1},\gamma'}\left(e_0^{\rm B},
e_1^{\rm B}\right)
={\mathfrak f}_{\sigma}^{z,\gamma}\left(e_0^{\rm B},
e_\infty^{\rm B}\right) \cdot {\bf exp}\left(\dfrac{1-\chi(\sigma)}{2} e_0^{\rm B}\right) \quad \left(\sigma \in G_K\right)
\end{align}
of $\ell$-adic Galois 1-cocycles along the path composition (\ref{gamma'}),
where $\chi: G_K \to \mathbb{Z}_{\ell}^\times$ denotes the $\ell$-adic cyclotomic character.
See subsection \ref{lmpoly} for the definition of ${\mathfrak f}^{z,\gamma} \in Z^1\left(G_K, \mathbb{Q}_{\ell} \langle \langle e_0^{\mathrm{B}},
e_1^{\mathrm{B}} \rangle \rangle^{\times}\right)$
and section \ref{pfsec} for the details of the proof.

\begin{thm}[The Landen formula for $\ell$-adic Galois multiple polylogarithms] \label{M2}
Let $z$ be a $K$-rational point of $\mathbb{P}^1 \backslash \{0,1,\infty\}$.
Given a path $\gamma \in \pi_1^{\rm top}\left(\mathbb{P}^1(\mathbb{C}) \backslash \{0,1,\infty\}; \overrightarrow{01}, z\right)$, define a path $\gamma'$ associated with $\gamma$ as in (\ref{gamma'}).
We regard topological paths $\gamma$ and $\gamma'$ as pro-$\ell$ \'etale paths on $\mathbb{P}^1_{\overline{K}} \backslash \{0,1,\infty\}$
via the comparison map
\[
\pi_1^{\rm top}\left(\mathbb{P}^1(\mathbb{C}) \backslash \{0,1,\infty\}; \overrightarrow{01}, \ast\right) \hookrightarrow \pi_1^\ellet\left(\mathbb{P}^1_{\overline{K}} \backslash \{0,1,\infty\};\overrightarrow{01},{\ast}\right)
\] 
induced by the inclusion $\overline{K} \hookrightarrow \mathbb{C}$
for $\ast \in \left\{z,\dfrac{z}{z-1}\right\}$.
For any $\sigma \in G_K$ and $\mathbf{k} \in \bigcup_{d=1}^{\infty} \mathbb{N}^d$, the following holds:
\begin{align} \label{Main2}
&Li^{\ell}_{\mathbf{k}}\left(\dfrac{z}{z-1};\gamma',\sigma\right)+(-1)^{1+\depth(\mathbf{k})} \sum_{\mathbf{J} \succeq \mathbf{k}} Li^{\ell}_{\mathbf{J}}\left(z;\gamma,\sigma\right)
\\
&\quad=
\sum_{M=1}^{\depth(\mathbf{k})} (-1)^{M+\depth(\mathbf{k})} \left(\sum_{\substack{\left(\mathbf{S}, \{(\mathbf{u}_{i}, r_i)\}_{i=1}^{M}\right) \in \mathrm{Part}(\mathbf{k}, M), \\ \exists i, (\mathbf{u}_{i}, r_i) \neq ((1), 0)}} \left(\sum_{\mathbf{J} \succeq \mathbf{S}} Li^{\ell}_{\mathbf{J}}\left({z};\gamma,\sigma\right)\right) \prod_{i=1}^M \mu_{\mathbf{u}_{i}, r_i}\right).\notag
\end{align}
Here,
\[
\mathrm{Part}(\mathbf{k}, M) := \left\{ 
\left(\mathbf{S}, \{(\mathbf{u}_{i}, r_i)\}_{i=1}^{M}\right) 
\ \middle| \ 
\begin{aligned}
&\quad \mathbf{S}=(S_1,\dots,S_M) \in \mathbb{N}^M, \\
&\quad (\mathbf{u}_{1}, r_1),\ldots,(\mathbf{u}_{M}, r_M) \in \bigcup_{d=1}^{\infty} \mathbb{N}^d \times \mathbb{Z}_{\ge 0}, \\
& \left( e_0^{S_M-1} W_{\mathbf{u}_{M}} e_0^{r_M} \right) \cdots \left( e_0^{S_1-1} W_{\mathbf{u}_{1}} e_0^{r_1} \right)= W_{\mathbf{k}}~\text{in}~\{e_0,e_1\}^{\ast}
\end{aligned}
\right\}
\]
where $\{e_0,e_1\}^{\ast}$ is the non-commutative free monoid generated by the non-commuting indeterminates $e_0$ and $e_1$.
Also, for $\mathbf{u} = (u_1, \dots, u_m) \in \mathbb{N}^m$ and $r \in \mathbb{Z}_{\ge 0}$, let
\[
W_{\mathbf{u}} := e_0^{u_m-1}e_1 e_0^{u_{m-1}-1}e_1 \cdots e_0^{u_1-1}e_1 \in \{e_0,e_1\}^{\ast}
\]
and 
\[
\mu_{\mathbf{u}, r} \in \mathbb{Q}
\]
denote the coefficient of the term $W_{\mathbf{u}} e_0^r \in \{e_0,e_1\}^{\ast}$ in the Baker-Campbell-Hausdorff sum ${\bf log}(\mathbf{exp}(-e_1)\mathbf{exp}(-e_0))$.
We express this coefficient explicitly as an integral involving Goldberg polynomials
$G_v(t)~(v \in \mathbb{Z}_{\geq1})$ as follows (see Definition~\ref{def:goldberg_poly} for the definition of $G_v(t)$):
\begin{align}\label{GPint}
\mu_{\mathbf{u}, r} = \sigma_{\mathbf{u}, r} \int_0^1 t^{\lfloor L/2 \rfloor} (t-1)^{\lfloor (L-1)/2 \rfloor} \prod_{i=1}^L G_{v_i}(t) \, dt.
\end{align}
Here, let $(v_1, \dots, v_L)$ be the multi-index obtained from the sequence $(u_m-1, 1, \dots, u_1-1, 1, r)$ by recursively contracting any triplet $(a, 0, b)$ into $a+b$ and then removing all leading and trailing zeros.
The sign factor $\sigma_{\mathbf{u}, r}$ is defined by
\[
\sigma_{\mathbf{u}, r} := \begin{cases}
-1 & (\text{if } u_m > 1), \\
(-1)^{r+\sum_{j=1}^{m}u_j} & (\text{if } u_m = 1).
\end{cases}
\]
\end{thm}

\begin{thm}[The Landen formula for $\ell$-adic Galois multiple polylogarithms in the case of $\mathbf{k}=(n)$]\label{M3}
Let the notation and assumptions be as in Theorem \ref{M2}.
For any $\sigma \in G_K$ and $n \in \mathbb{N}$, the following holds:
\begin{align} \label{Main3}
~~{Li}^{\ell}_n\left(\frac{z}{z-1};\gamma',\sigma\right)
+\sum_{\substack{\mathbf{J} \in \mathbb{N}^d,~d \in \mathbb{N} \\ {\rm wt}(\mathbf{J})=n}} {Li}^{\ell}_{\mathbf{J}}(z;\gamma,\sigma) = \sum_{m=1}^{n-1} \frac{(-1)^{n-m+1}}{(n-m+1)!} {Li}^{\ell}_m\left(\frac{z}{z-1};\gamma',\sigma\right).
\end{align}
Here, for $\mathbf{J}=(j_1,\cdots,j_d)$, let
\[
\mathrm{wt}(\mathbf{J}):=j_1+\cdots+j_d.
\]
\end{thm}

\begin{rem}\label{rem:weight_decrease}
On the left-hand side of the formula \eqref{Main2}, 
every term $Li^{\ell}_{\mathbf{J}}(z;\gamma,\sigma)$ satisfies the condition $\mathrm{wt}(\mathbf{J}) = \mathrm{wt}(\mathbf{k})$
by the definition of $\mathbf{J} \succeq \mathbf{k}$ (see Definition \ref{saibun}).
On the other hand,
on the right-hand side of the formula \eqref{Main2}, 
every term $Li^{\ell}_{\mathbf{J}}(z;\gamma,\sigma)$ satisfies the condition $\mathrm{wt}(\mathbf{J}) < \mathrm{wt}(\mathbf{k})$.
Indeed, 
from the definition of $\mathrm{Part}(\mathbf{k}, M)$, 
the total weight of $\mathbf{k}$ is given by
\[
\mathrm{wt}(\mathbf{k}) = \mathrm{wt}(\mathbf{S}) + \sum_{i=1}^M (\mathrm{wt}(\mathbf{u}_i) + r_i - 1).
\]
Since $(\mathbf{u}_i, r_i) = ((1), 0)$ corresponds to the leading terms of the Baker-Campbell-Hausdorff sum ${\bf log}(\mathbf{exp}(-e_1)\mathbf{exp}(-e_0))$,
the condition that at least one pair $(\mathbf{u}_i, r_i)$ is not $((1), 0)$ implies the strict inequality
\[
\sum_{i=1}^M (\mathrm{wt}(\mathbf{u}_i) + r_i - 1) > 0.
\]
These lower weight terms 
specific to the $\ell$-adic Galois setting
are referred to in \cite[Subsection 4.3]{NW12} as ``$\ell$-adic error terms''.
\end{rem}

\begin{rem}
The appearance of lower weight terms in the $\ell$-adic Landen formula (Theorem \ref{M2}) stands in sharp contrast to the complex case (Theorem \ref{M1}) where the weights of the terms are strictly preserved. 
This structural discrepancy stems solely from the different relations satisfied by the variable at $\infty$. 
In the complex setting,
the variable $e_{\infty}^{\rm dR}$ in the complex KZ solution (\ref{crint1}) obeys the linear relation
\[
e_{\infty}^{\rm dR} = -e_1^{\rm dR} - e_0^{\rm dR}
\]
reflecting the sum of residues
$$\text{\rm Res}_{z=0}(\omega) + \text{\rm Res}_{z=1}(\omega) + \text{\rm Res}_{z=\infty}(\omega) = 0$$
where $\omega:=\frac{dz}{z}e_0^{\rm dR}+\frac{dz}{z-1}e_1^{\rm dR}$ is the canonical 1-form on $\mathbb{P}^1(\mathbb{C}) \backslash \{0,1,\infty\}$.
On the other hand, in the $\ell$-adic Galois setting,
the variable
$e_{\infty}^{\rm B}$
in the $\ell$-adic Galois 1-cocycle (\ref{crint2}) is governed by the group-theoretic relation $l_0 l_1 l_{\infty} = 1$ in the topological fundamental group
\[
\pi_1^{\rm top}\left(\mathbb{P}^1(\mathbb{C}) \backslash \{0,1,\infty\}, \overrightarrow{01}\right)=\left<l_0,
l_1,
l_{\infty} \mid l_0 l_1 l_{\infty}=1 \right>,
\]
and forms the Baker-Campbell-Hausdorff sum
\begin{align*}
e_{\infty}^{\rm B} &= {\bf log}\left({\bf exp}\left(-e_1^{\rm B}\right){\bf exp}\left(-e_0^{\rm B}\right)\right)\\
&=-e_1^{\rm B}-e_0^{\rm B}+\text{$($higher-order terms$)$}.
\end{align*}
Consequently, the ``error terms'' in the $\ell$-adic Landen formula are precisely the manifestations of the higher-order terms in the BCH sum reflecting the non-abelian nature of the pro-$\ell$ \'etale fundamental group $\pi_1^\ellet\left(\mathbb{P}^1_{\overline{K}} \backslash \{0,1,\infty\},\overrightarrow{01}\right)$.
Our result (\ref{Main2}) reveals that the Goldberg integral (\ref{GPint}) plays an important role in achieving an explicit understanding of the action of $G_K$ on 
the pro-$\ell$ \'etale fundamental groupoid of $\mathbb{P}^1_{\overline{K}} \backslash \{0,1,\infty\}$ with $K$-rational base points. 
\end{rem}

{\renewcommand{\arraystretch}{1.55}
\tabcolsep = 0.3cm
\small
\begin{table}[hbtp]
\centering
\caption{Formulae to be proved}
\label{funcs}
  \begin{tabular}{|c||c|}  \hline
    $\ell$-adic Galois side & Complex side  \\ \hline \hline

   \tiny$Li_{\mathbf{k}}^\ell(z;\gamma,\sigma) \in \bQ_\ell$ & \tiny $Li_{\mathbf{k}}(z;\gamma) \in \bC$ \\

   \tiny where $z$ is a $K$-rational base point of $\mathbb{P}^1 \backslash \{0,1,\infty\}$ & \tiny where $z$ is a $\bC$-rational base point of $\mathbb{P}^1 \backslash \{0,1,\infty\}$ \\

   \tiny and $(\gamma,\sigma) \in \pi_1^\ellet \left(\mathbb{P}^1_{\overline{K}} \backslash \{0,1,\infty\};\overrightarrow{01},{z}\right) \times G_K$. & \tiny and $\gamma \in \pi_1^{\rm top}\left(\mathbb{P}^1(\mathbb{C}) \backslash \{0,1,\infty\}; \overrightarrow{01}, z\right)$. \\ \hline \hline
    

    \tiny$\displaystyle Li^{\ell}_{\mathbf{k}}\left(\dfrac{z}{z-1};\gamma',\sigma\right)+(-1)^{1+\depth(\mathbf{k})} \sum_{\mathbf{J} \succeq \mathbf{k}} Li^{\ell}_{\mathbf{J}}\left(z;\gamma,\sigma\right)
		$ & \tiny$\displaystyle Li_{\mathbf{k}}\left(\dfrac{z}{z-1};\gamma'\right)+(-1)^{1+\depth(\mathbf{k})} \sum_{\mathbf{J} \succeq \mathbf{k}} Li_{\mathbf{J}}\left(z;\gamma\right)=0$ \\ 

     \tiny $\displaystyle =
\sum_{M=1}^{\depth(\mathbf{k})} (-1)^{M+\depth(\mathbf{k})} \left(\sum_{\mathrm{Part}(\mathbf{k}, M)^\ast} \left(\sum_{\mathbf{J} \succeq \mathbf{S}} Li^{\ell}_{\mathbf{J}}\left({z};\gamma,\sigma\right)\right) \prod_{i=1}^M \mu_{\mathbf{u}_{i}, r_i}\right)$ & \cite[Proposition 9]{OU04}  \\

Theorem \ref{M2} & Theorem \ref{M1} \\ \hline \hline

   \tiny $\displaystyle {Li}^{\ell}_n\left(\frac{z}{z-1};\gamma',\sigma\right)
+\sum_{\mathbf{J} \succeq n} {Li}^{\ell}_{\mathbf{J}}(z;\gamma,\sigma)$  &  \\ 

   \tiny $\displaystyle = \sum_{m=1}^{n-1} \frac{(-1)^{n-m+1}}{(n-m+1)!} {Li}^{\ell}_m\left(\frac{z}{z-1};\gamma',\sigma\right)$  & \tiny $\displaystyle {Li}^{}_n\left(\frac{z}{z-1};\gamma'\right)
+\sum_{\mathbf{J} \succeq n} {Li}^{}_{\mathbf{J}}(z;\gamma)=0$ \\

   \tiny $\displaystyle = \sum_{j=1}^{n-1} \frac{(-1)^{j+1} B_{j}}{j!} \left(\sum_{\mathbf{J} \succeq n-j} {Li}^{\ell}_{\mathbf{J}}(z;\gamma,\sigma)\right)$  &  \\

Theorem \ref{M3} & Theorem \ref{M1}~~~($\mathbf{k}=n$) \\ \hline

\tiny $\displaystyle Li^{\ell}_{2}\left(\dfrac{z}{z-1};\gamma',\sigma\right) + Li^{\ell}_{2}(z;\gamma,\sigma) + Li^{\ell}_{1,1}(z;\gamma,\sigma)$  &  \\ 

   \tiny $\displaystyle =\frac{1}{2} Li^{\ell}_1\left(\frac{z}{z-1};\gamma',\sigma\right)$  & \tiny $\displaystyle Li^{}_{2}\left(\dfrac{z}{z-1};\gamma'\right) + Li^{}_{2}(z;\gamma) + Li^{}_{1,1}(z;\gamma)=0$ \\

   \tiny $\displaystyle = -\frac{1}{2} Li^{\ell}_{1}(z;\gamma,\sigma)$  &  \\

Proposition \ref{sceq2} & Theorem \ref{M1}~~~($\mathbf{k}=2$) \\ \hline

\tiny $\displaystyle Li^{\ell}_{3}\left(\dfrac{z}{z-1};\gamma',\sigma\right) +\sum_{\mathbf{J} \succeq 3} {Li}^{\ell}_{\mathbf{J}}(z;\gamma,\sigma)$  &  \\ 

   \tiny $\displaystyle = -\frac{1}{6} Li^{\ell}_1\left(\frac{z}{z-1};\gamma',\sigma\right) + \frac{1}{2} Li^{\ell}_2\left(\frac{z}{z-1};\gamma',\sigma\right)$  & \tiny $\displaystyle Li^{}_{3}\left(\dfrac{z}{z-1};\gamma'\right) +\sum_{\mathbf{J} \succeq 3} {Li}^{}_{\mathbf{J}}(z;\gamma)=0$ \\

   \tiny $\displaystyle = -\frac{1}{12} Li^{\ell}_{1}(z;\gamma,\sigma) - \frac{1}{2} Li^{\ell}_{2}(z;\gamma,\sigma) - \frac{1}{2} Li^{\ell}_{1,1}(z;\gamma,\sigma)$  &   \\

Proposition \ref{sceq3} & Theorem \ref{M1}~~~($\mathbf{k}=3$) \\ \hline

\tiny $\displaystyle Li^{\ell}_{4}\left(\dfrac{z}{z-1};\gamma',\sigma\right) + \sum_{\substack{\mathbf{J} \in \mathbb{N}^d,~d \in \mathbb{N} \\ {\rm wt}(\mathbf{J})=4}} Li^{\ell}_{\mathbf{J}}(z;\gamma,\sigma)$  &  \\ 

   \tiny $\displaystyle = -\frac{1}{2} \left( Li^{\ell}_{3}(z;\gamma,\sigma) + Li^{\ell}_{2,1}(z;\gamma,\sigma) + Li^{\ell}_{1,2}(z;\gamma,\sigma)\right.$  & \tiny $\displaystyle {Li}^{}_4\left(\frac{z}{z-1};\gamma'\right)
+\sum_{\mathbf{J} \succeq 4} {Li}^{}_{\mathbf{J}}(z;\gamma)=0$ \\

   \tiny $\displaystyle \quad \quad \quad \left.+ Li^{\ell}_{1,1,1}(z;\gamma,\sigma) \right) - \frac{1}{12} \left( Li^{\ell}_{2}(z;\gamma,\sigma) + Li^{\ell}_{1,1}(z;\gamma,\sigma) \right) $  &  \\

Proposition \ref{sceq4} & Theorem \ref{M1}~~~($\mathbf{k}=4$) \\ \hline

\tiny $\displaystyle Li^{\ell}_{2,3}\left(\frac{z}{z-1};\gamma',\sigma\right)-
\sum_{\mathbf{J}\succeq(2,3)} Li^{\ell}_{\mathbf{J}}(z;\gamma,\sigma)$  &  \\ 

   \tiny $\displaystyle = -\frac{1}{120} Li^{\ell}_1(z;\gamma,\sigma) - \frac{1}{12} Li^{\ell}_2(z;\gamma,\sigma) - \frac{1}{24} Li^{\ell}_{1,1}(z;\gamma,\sigma)$  & \\

   \tiny $\displaystyle - \frac{1}{6} Li^{\ell}_3(z;\gamma,\sigma) - \frac{1}{12} Li^{\ell}_{2,1}(z;\gamma,\sigma) - \frac{1}{12} Li^{\ell}_{1,2}(z;\gamma,\sigma)$  & \tiny $\displaystyle Li^{}_{2,3}\left(\frac{z}{z-1};\gamma'\right)-
\sum_{\mathbf{J}\succeq(2,3)} Li^{}_{\mathbf{J}}(z;\gamma)=0$ \\

   \tiny $\displaystyle + \frac{1}{2} \left( Li^{\ell}_{2,2}(z;\gamma,\sigma) + Li^{\ell}_{2,1,1}(z;\gamma,\sigma) \right.$  &  \\

   \tiny $\displaystyle \left. + Li^{\ell}_{1,1,2}(z;\gamma,\sigma) + Li^{\ell}_{1,1,1,1}(z;\gamma,\sigma) \right)$  &  \\

Proposition \ref{sceq23} & Theorem \ref{M1}~~~($\mathbf{k}=(2,3)$) \\ \hline
   
  \end{tabular}
\end{table}}

\section{Key theorems on non-commutative formal power series in two variables}

In this section, let $K$ be a field of characteristic $0$.
Let $K\langle\langle e_0, e_1 \rangle\rangle$ be the non-commutative formal power series ring in two non-commuting variables $e_0$ and $e_1$ over $K$.
We define $e_\infty$ by
\[
e_\infty := -e_1-e_0.
\]
Furthermore, let $e'_\infty$ denote the Baker-Campbell-Hausdorff sum given by
\begin{align*}
e'_\infty &:= \mathbf{log}(\mathbf{exp}(-e_1)\mathbf{exp}(-e_0))\\
&=\sum_{m=1}^{\infty} \frac{(-1)^{m-1}}{m}\left(\left(\sum_{n=0}^{\infty}\frac{1}{n!}{(-e_1)}^n\right)\left(\sum_{l=0}^{\infty}\frac{1}{l!}{(-e_0)}^l\right)-1\right)^m \\
&=-e_1-e_0+(\text{higher-order terms}).
\end{align*}
Our goal in this section is to investigate the coefficient of the monomial
$$W_{\mathbf{k}} := e_0^{k_d-1}e_1 e_0^{k_{d-1}-1}e_1 \cdots e_0^{k_1-1}e_1$$
in the expansion of
\[
f(e_0,e_\infty),~f(e_0,e'_\infty) \in K\langle\langle e_0, e_1 \rangle\rangle
\]
for arbitrary $f(e_0,e_1) \in K\langle\langle e_0, e_1 \rangle\rangle$ and $\mathbf{k}=(k_1,\cdots,k_d) \in \mathbb{N}^d$.
We begin by establishing some notation.

\begin{dfn}[Weight and depth of multi-indices]
An element of $\bigcup_{d=1}^{\infty} \mathbb{N}^d$ is called a multi-index.
For a multi-index $\mathbf{k} = (k_1, \dots, k_d) \in \mathbb{N}^d$, we define its weight and depth as follows:
\[
\mathrm{wt}(\mathbf{k}) := \sum_{i=1}^d k_i, \quad \depth(\mathbf{k}) := d.
\]
\end{dfn}

\begin{dfn}[Refinement of multi-indices]\label{saibun}
For multi-indices $\mathbf{k}=(k_1, \dots, k_d) \in \mathbb{N}^d$ and $\mathbf{J}=(j_1, \dots, j_l) \in \mathbb{N}^l$, we say that $\mathbf{J}$ is a refinement of $\mathbf{k}$,
denoted by
\[
\mathbf{J} \succeq \mathbf{k}
\]
if $\mathbf{k}$ is obtained by partial sums of the components of $\mathbf{J}$ preserving the order.
More precisely, $\mathbf{J} \succeq \mathbf{k}$ if there exists a multi-index $(m_1, \dots, m_d) \in \mathbb{N}^d$ such that $\sum_{i=1}^d m_i = l$ and for any $i=1, \dots, d$,
\[
k_i = \sum_{q=1}^{m_i} j_{S_{i-1} + q}
\]
where $S_0 := 0$ and $S_i := \sum_{r=1}^i m_r$.

\end{dfn}

\begin{dfn}[Monomial $W_{\mathbf{k}}$ and coefficient $\operatorname{Coeff}_{\mathbf{k}}$]
For a multi-index $\mathbf{k} = (k_1, \dots, k_d) \in \mathbb{N}^d$, we define the corresponding non-commutative monomial $W_{\mathbf{k}}$ as follows:
\[
W_{\mathbf{k}} := e_0^{k_d-1}e_1 e_0^{k_{d-1}-1}e_1 \cdots e_0^{k_1-1}e_1.
\]
For any $f(e_0, e_1) \in K\langle\langle e_0, e_1 \rangle\rangle$, we denote by 
\[
\operatorname{Coeff}_{\mathbf{k}}(f(e_0, e_1)) \in K
\]
the coefficient of the term $W_{\mathbf{k}}$ in $f(e_0, e_1)$.
\end{dfn}

\begin{rem}
Any formal power series $f(e_0, e_1) \in K\langle\langle e_0, e_1 \rangle\rangle$ can be uniquely written in the form
\[
f(e_0, e_1) = f(0,0)+\sum_{\mathbf{k}} \operatorname{Coeff}_{\mathbf{k}}(f(e_0, e_1)) W_{\mathbf{k}} + (\text{terms ending in } e_0).
\]
\end{rem}

\begin{thm}[Coefficient of $W_{\mathbf{k}}$ in $f(e_0, e_\infty)$]\label{thm22}
Let $f(e_0, e_1) \in K\langle\langle e_0, e_1 \rangle\rangle$ and $\mathbf{k}
\in \mathbb{N}^{{\rm dp}(\mathbf{k})}$.
The coefficient $\operatorname{Coeff}_{\mathbf{k}}(f(e_0, e_\infty))$ of the term $W_{\mathbf{k}}$
in the expansion of $f(e_0, e_\infty)\left(=f\left(e_0, -e_1-e_0\right)\right) \in K\langle\langle e_0, e_1 \rangle\rangle$ is given by
\[
\operatorname{Coeff}_{\mathbf{k}}(f(e_0, e_\infty)) = \sum_{\mathbf{J} \succeq \mathbf{k}} (-1)^{\depth(\mathbf{J})} \operatorname{Coeff}_{\mathbf{J}}(f(e_0, e_1)).
\]
In particular, for $n \in \mathbb{N}$, 
\[
\operatorname{Coeff}_{n}(f(e_0, e_\infty)) = \sum_{\substack{\mathbf{J} \in \mathbb{N}^{l},~l \in \mathbb{N}, \\ {\rm wt}(\mathbf{J})=n}} (-1)^{\depth(\mathbf{J})} \operatorname{Coeff}_{\mathbf{J}}(f(e_0, e_1)).
\]
\end{thm}

\begin{proof}
Let $\varphi: K\langle\langle e_0, e_1 \rangle\rangle \to K\langle\langle e_0, e_1 \rangle\rangle$ be the ring homomorphism defined by
$
\varphi(e_0) = e_0$ and $\varphi(e_1) = e_\infty~(=-e_1-e_0)$.
Writing
\[
f(e_0, e_1) = f(0,0)+\sum_{\mathbf{J}} \operatorname{Coeff}_{\mathbf{J}}(f(e_0, e_1)) W_{\mathbf{J}} + (\text{terms ending in } e_0),
\]
we have
\[
f(e_0, e_\infty) = f(0,0)+\sum_{\mathbf{J}} \operatorname{Coeff}_{\mathbf{J}}(f(e_0, e_1)) \varphi(W_{\mathbf{J}}) + \varphi(\text{terms ending in } e_0).
\]
Since $\varphi(e_0) = e_0$,
the part $\varphi(\text{terms ending in } e_0)$ consists of terms ending in $e_0$. Therefore, to compute $\operatorname{Coeff}_{\mathbf{k}}(f(e_0, e_\infty))$, it suffices to identify the coefficient of $W_{\mathbf{k}}$ in the expansion of $\sum_{\mathbf{J}} \operatorname{Coeff}_{\mathbf{J}}(f(e_0, e_1)) \varphi(W_{\mathbf{J}})$.
Let $\mathbf{J} = (j_1, \dots, j_l) \in \mathbb{N}^l$. 
Then,
\begin{align*}
\varphi(W_{\mathbf{J}}) = (-1)^{\depth(\mathbf{J})} \prod_{i=l}^{1} (e_0^{j_i} + e_0^{j_i-1}e_1)
\end{align*}
where the product is ordered from $i=l$ down to $i=1$. Thus,
\[
\varphi\left(\operatorname{Coeff}_{\mathbf{J}}(f(e_0, e_1)) W_{\mathbf{J}}\right) = (-1)^{\depth(\mathbf{J})} \operatorname{Coeff}_{\mathbf{J}}(f(e_0, e_1)) \Pi_{\mathbf{J}}
\]
where
\[
\Pi_{\mathbf{J}} := \prod_{i=l}^{1} (e_0^{j_i} + e_0^{j_i-1}e_1).
\]
Expanding $\Pi_{\mathbf{J}}$, we obtain
\[
\Pi_{\mathbf{J}} = \sum_{A \subseteq \{1, \dots, l\}} \Theta_{A} \quad \text{where}\quad \Theta_{A} := \prod_{i=l}^{1} \theta_{i}^{(A)}\quad \text{and}\quad \theta_{i}^{(A)} := \begin{cases} e_0^{j_i-1}e_1 & (i \in A), \\ e_0^{j_i} & (i \notin A). \end{cases}
\]
The monomial $\Theta_{A}$ equals $W_{\mathbf{k}} = e_0^{k_d-1}e_1 \cdots e_0^{k_1-1}e_1$ if and only if there exists a multi-index $(m_1, \dots, m_d) \in \mathbb{N}^d$ such that $\sum_{i=1}^d m_i = l$ and
\[
k_i = \sum_{q=1}^{m_i} j_{S_{i-1} + q} \quad \left(S_i := \sum_{r=1}^i m_r\right)
\]
for all $i=1, \dots, d$ hold
and the set $A$ is given by
\[
A=\left\{1,~1+m_1,~1+m_1+m_2, \dots, 1+\sum_{t=1}^{d-1}m_t\right\}.
\]
This condition is equivalent to $\mathbf{J} \succeq \mathbf{k}$, and such a set $A$ is unique for a given refinement.
Therefore, the term $W_{\mathbf{k}}$ appears in the expansion of $\Pi_{\mathbf{J}}$ if and only if $\mathbf{J} \succeq \mathbf{k}$, with coefficient $1$.
Consequently,
\begin{align*}
\operatorname{Coeff}_{\mathbf{k}}(f(e_0, e_\infty))
&= \operatorname{Coeff}_{\mathbf{k}}\left(\sum_{\mathbf{J}} (-1)^{\depth(\mathbf{J})} \operatorname{Coeff}_{\mathbf{J}}(f(e_0, e_1)) \Pi_{\mathbf{J}}\right) \\
&= \sum_{\mathbf{J} \succeq \mathbf{k}} (-1)^{\depth(\mathbf{J})} \operatorname{Coeff}_{\mathbf{J}}(f(e_0, e_1)).
\end{align*}
\end{proof}

Next, we consider the Baker-Campbell-Hausdorff sum
\[
e'_\infty := \mathbf{log}(\mathbf{exp}(-e_1)\mathbf{exp}(-e_0)).
\]

\begin{dfn}[The coefficient $\mu_{\mathbf{u}, r}$ of the BCH sum $e'_\infty$]
For $\mathbf{u}=(u_1,\ldots,u_m) \in \mathbb{N}^{m}$ and $r \in \mathbb{Z}_{\ge 0}$,
let
\[
\mu_{\mathbf{u}, r}:=\operatorname{Coeff}_{W_{\mathbf{u}} e_0^r}\left(e'_\infty\right) \in \mathbb{Q}
\]
denote the coefficient of the term $W_{\mathbf{u}} e_0^r$
in the expansion of $e'_\infty$.
That is,
\begin{align*}
e'_\infty &= -e_0 +\sum_{\substack{\mathbf{u}=(u_1,\ldots,u_m) \in \mathbb{N}^{m},\\m \in \mathbb{N},~r \in \mathbb{Z}_{\ge 0}}} \mu_{\mathbf{u}, r} W_{\mathbf{u}} e_0^r\\
&= -e_0 +\sum_{\substack{\mathbf{u}=(u_1,\ldots,u_m) \in \mathbb{N}^{m},\\m \in \mathbb{N},~r \in \mathbb{Z}_{\ge 0}}} \mu_{\mathbf{u}, r} \left(e_0^{u_m-1}e_1 e_0^{u_{m-1}-1}e_1 \cdots e_0^{u_1-1}e_1\right) e_0^r.
\end{align*}
Note that the coefficient corresponding to the leading term $-e_1$ is $\mu_{(1), 0} = -1$.
\end{dfn}

To provide an explicit formula for $\mu_{\mathbf{u}, r} \in \mathbb{Q}$, we use the Goldberg polynomial
introduced by Goldberg \cite[THEOREM 1]{G56}.

\begin{dfn}[Goldberg polynomials]\label{def:goldberg_poly}
For an integer $v>0$, we define the polynomial $G_v(t) \in \mathbb{Q}[t]$ by the following recurrence relation:
\[
G_1(t) = 1,\quad v G_v(t) = \frac{d}{dt} \left( t(t-1) G_{v-1}(t) \right)\quad (v \ge 2).
\]
\end{dfn}

\begin{ex}[Goldberg polynomials for $v=2,\cdots,6$]
\begin{align*}
&G_2(t) = t - \frac{1}{2}, \quad
G_3(t) = t^2 - t + \frac{1}{6}, \quad
G_4(t) = t^3 - \frac{3}{2}t^2 + \frac{7}{12}t - \frac{1}{24},\\
&G_5(t) = t^4 - 2t^3 + \frac{5}{4}t^2 - \frac{1}{4}t + \frac{1}{120},\quad 
G_6(t) = t^5 - \frac{5}{2}t^4 + \frac{13}{6}t^3 - \frac{3}{4}t^2 + \frac{31}{360}t - \frac{1}{720}.
\end{align*}
\end{ex}

\noindent
With these preparations, $\mu_{\mathbf{u}, r}$ is given by the following integral formula. 

\begin{prop}[Explicit formula for $\mu_{\mathbf{u}, r}$]\label{prop:integral_rep_corrected}
For $\mathbf{u} = (u_1, \dots, u_m) \in \mathbb{N}^m$ and $r \in \mathbb{Z}_{\ge 0}$,
let
\[
(v_1, \dots, v_L)
\]
be the multi-index obtained from the sequence $(u_m-1, 1, \dots, u_1-1, 1, r)$ by recursively contracting any triplet $(a, 0, b)$ into $a+b$ and then removing all leading and trailing zeros.
Then, the coefficient $\mu_{\mathbf{u}, r}$ of the term $W_{\mathbf{u}}e_0^r$ in $e'_\infty=\mathbf{log}(\mathbf{exp}(-e_1)\mathbf{exp}(-e_0))$ is given by
\[
\mu_{\mathbf{u}, r} = \sigma_{\mathbf{u}, r} \int_0^1 t^{\lfloor L/2 \rfloor} (t-1)^{\lfloor (L-1)/2 \rfloor} \prod_{i=1}^L G_{v_i}(t) \, dt
\]
where the sign factor $\sigma_{\mathbf{u}, r}$ is defined by
\[
\sigma_{\mathbf{u}, r} = \begin{cases}
-1 & (\text{if } u_m > 1), \\
(-1)^{N} & (\text{if } u_m = 1)
\end{cases}
\]
with $N := \mathrm{wt}(\mathbf{u}) + r$ being the total weight of the monomial $W_{\mathbf{u}} e_0^r$.
\end{prop}

\begin{proof}
Set $X=-e_1$ and $Y=-e_0$.
Then $e'_\infty = \mathbf{log}(\mathbf{exp}(-e_1)\mathbf{exp}(-e_0)) = \mathbf{log}(\mathbf{exp}(X)\mathbf{exp}(Y))$ and
\[
W_{\mathbf{u}}e_0^r = (-1)^N Y^{u_m-1} X \cdots Y^{u_1-1} X Y^r.
\]
According to Goldberg \cite[THEOREM 1]{G56}, the coefficient $c_x$ of the monomial starting with $X$ and having the exponent sequence $(v_1, \dots, v_L)$, i.e., $X^{v_1}Y^{v_2}X^{v_3}\cdots$, in the expansion of $\mathbf{log}(\mathbf{exp}(X)\mathbf{exp}(Y))$ is given by the integral
\[
I := \int_0^1 t^{\lfloor L/2 \rfloor} (t-1)^{\lfloor (L-1)/2 \rfloor} \prod_{i=1}^L G_{v_i}(t) \, dt,
\]
and the coefficient $c_y$ of the monomial starting with $Y$ satisfies $c_y = (-1)^{N-1} c_x$.

Case $u_m > 1$: The monomial $W_{\mathbf{u}}e_0^r$ starts with $e_0$, i.e., $Y$.
The required coefficient is $(-1)^N c_y$, so
\[
\mu_{\mathbf{u}, r} = (-1)^N c_y = (-1)^N (-1)^{N-1} I = (-1)^{2N-1} I = -I.
\]

Case $u_m = 1$: The monomial $W_{\mathbf{u}}e_0^r$ starts with $e_1$, i.e., $X$.
The required coefficient is $(-1)^N c_x$, so
\[
\mu_{\mathbf{u}, r} = (-1)^N c_x = (-1)^N I.
\]

Combining these gives the desired formula.
\end{proof}

\begin{ex}[Calculation of coefficients $\mu_{\mathbf{u}, r}$]
We illustrate two concrete examples of the formula.

\begin{enumerate}
    \item {\bf Case $\mathbf{u}=(3)$ and $r=0$ $($Weight $3$, Depth $1$$)$:} \\
    We consider the coefficient of the monomial $W_{(3)} = e_0^{3-1}e_1 = e_0^2 e_1$.
    The sequence of exponents is derived from $(u_1-1, 1, r) = (2, 1, 0)$. Removing zeros, we get
$(2, 1)$, so $L=2$.
    The formula involves $G_2(t)$ and $G_1(t)$:
    \[
    \int_0^1 t^{\lfloor 2/2 \rfloor} (t-1)^{\lfloor 1/2 \rfloor} G_2(t) G_1(t) \, dt
    = \int_0^1 t \left(t-\frac{1}{2}\right) \, dt
    = \frac{1}{12}.
    \]
    Since $u_m=u_1=3 > 1$, the sign is $\sigma_{(3), 0} = -1$. Thus,
    \[
    \mu_{(3), 0} = - \frac{1}{12}.
    \]
    This matches the coefficient of $e_0^2 e_1$ in the expansion of $e'_\infty=\mathbf{log}(\mathbf{exp}(-e_1)\mathbf{exp}(-e_0))$.

    \item {\bf Case $\mathbf{u}=(2, 1)$ and $r=0$ $($Weight $3$, Depth $2$$)$:} \\
    We consider the coefficient of the monomial
    \[
    W_{(2,1)} = e_0^{u_2-1}e_1 e_0^{u_1-1}e_1 = e_0^{1-1}e_1 e_0^{2-1}e_1 = e_1 e_0 e_1.
    \]
    The sequence is derived from $(u_2-1, 1, u_1-1, 1, r) = (0, 1, 1, 1, 0)$.
    Removing zeros yields $(1, 1, 1)$, so $L=3$.
    The integral becomes:
    \[
    \int_0^1 t^{\lfloor 3/2 \rfloor} (t-1)^{\lfloor 2/2 \rfloor} G_1(t)^3 \, dt
    = \int_0^1 t(t-1) \, dt
    = -\frac{1}{6}.
    \]
    Since $u_m = u_2 = 1$, the sign is determined by the total weight $N = 3$, giving $\sigma_{(2,1), 0} = (-1)^3 = -1$.
    Therefore,
    \[
    \mu_{(2,1), 0} = \frac{1}{6}.
    \]
    This matches the coefficient of $e_1 e_0 e_1$ in the expansion of $e'_\infty=\mathbf{log}(\mathbf{exp}(-e_1)\mathbf{exp}(-e_0))$.
\end{enumerate}
\end{ex}

\begin{prop}[Explicit formula for $\mu_{{(u)}, 0}$ in terms of Bernoulli numbers]\label{prop:integral_rep_corrected2}
For any $u \in \mathbb{N}$,
\[
\mu_{(u), 0} = \frac{(-1)^u B_{u-1}}{(u-1)!}
\]
holds.
Here, $B_{u-1}$ denotes the Bernoulli number with $B_1 = -1/2$.
\end{prop}

\begin{proof}
For $u=1$, the assertion follows directly from
\[
\mu_{(1),0}=-1=\frac{(-1)^1B_0}{0!}.
\]
Hence, in what follows, we may assume that $u\ge2$.
By definition, $\mu_{(u), 0}$ is the coefficient of $W_{u} = e_0^{u-1}e_1$ in $e'_\infty=\mathbf{log}(\mathbf{exp}(-e_1)\mathbf{exp}(-e_0))$. 

Setting $X = -e_1$ and $Y = -e_0$,
we have
\[
W_{u} = (-1)^u Y^{u-1}X.
\]
According to Goldberg \cite[THEOREM 3]{G56}, the coefficient $c_x(v_1, v_2)$ of the term $X^{v_1}Y^{v_2}$ in $e'_\infty=\mathbf{log}(\mathbf{exp}(X)\mathbf{exp}(Y))$ is
\[
c_x(v_1, v_2) = \frac{(-1)^{v_1}}{v_1! v_2!} \sum_{k=1}^{v_2} \binom{v_2}{k} B_{v_1+v_2-k}.
\]
Setting $v_1=u-1, v_2=1$ in this formula yields
\[
c_x(u-1, 1) = \frac{(-1)^{u-1}}{(u-1)! 1!} \sum_{k=1}^{1} \binom{1}{k} B_{(u-1)+1-k} = \frac{(-1)^{u-1}}{(u-1)!} B_{u-1}.
\]
The relationship between the coefficient $c_x$ of $X^{v_1}Y^{v_2}$ and the coefficient $c_y$ of $Y^{v_1}X^{v_2}$
is given in \cite[THEOREM 1]{G56} by
$c_y(v_1, v_2) = (-1)^{v_1+v_2-1} c_x(v_1, v_2)$.
Therefore, the coefficient of the term $Y^{u-1}X$ in $e'_\infty$ is
\[
c_y(u-1,1) = (-1)^{u-1} c_x(u-1, 1) = (-1)^{u-1} \frac{(-1)^{u-1}}{(u-1)!} B_{u-1} = \frac{B_{u-1}}{(u-1)!}.
\]
Thus, since $\mu_{(u), 0}$ is the coefficient of $e_0^{u-1}e_1$ in $e'_\infty= \mathbf{log}(\mathbf{exp}(X)\mathbf{exp}(Y))$, we have
\[
\mu_{(u), 0} = (-1)^u c_y(u-1,1) = \frac{(-1)^u B_{u-1}}{(u-1)!}.
\]
The factor $(-1)^u$ comes from the change of variables $X=-e_1$ and $Y=-e_0$.
\end{proof}

\begin{thm}[Coefficient of $W_{\mathbf{k}}$ in $f(e_0, e'_\infty)$]\label{thm33}
Let $f(e_0, e_1) \in K\langle\langle e_0, e_1 \rangle\rangle$ and $\mathbf{k}=(k_1, \dots, k_d) \in \mathbb{N}^d$.
The coefficient $\operatorname{Coeff}_{\mathbf{k}}(f(e_0, e'_\infty))$ of the term 
\[
W_{\mathbf{k}}=e_0^{k_d-1}e_1 e_0^{k_{d-1}-1}e_1 \cdots e_0^{k_1-1}e_1
\]
in the expansion of $f(e_0, e'_\infty) \in K\langle\langle e_0, e_1 \rangle\rangle$ is given by the following formula:
\begin{align*}
&\operatorname{Coeff}_{\mathbf{k}}(f(e_0, e'_\infty))+\sum_{\mathbf{J} \succeq \mathbf{k}} (-1)^{1+\depth(\mathbf{J})} \operatorname{Coeff}_{\mathbf{J}}(f(e_0, e_1))\\ 
&=\sum_{M=1}^{\depth(\mathbf{k})} (-1)^M \left(\sum_{\substack{\left(\mathbf{S}, \{(\mathbf{u}_{i}, r_i)\}_{i=1}^{M}\right) \in \mathrm{Part}(\mathbf{k}, M), \\ \exists i, (\mathbf{u}_{i}, r_i) \neq ((1), 0)}} \left(\sum_{\mathbf{J} \succeq \mathbf{S}} (-1)^{\depth(\mathbf{J})} \operatorname{Coeff}_{\mathbf{J}}(f(e_0, e_1))\right) \prod_{i=1}^M \mu_{\mathbf{u}_{i}, r_i}\right).
\end{align*}
Here, $\mathrm{Part}(\mathbf{k}, M)$ denotes the set of all generalized partitions of $\mathbf{k}$:
\[
\mathrm{Part}(\mathbf{k}, M) := \left\{ 
\left(\mathbf{S}, \{(\mathbf{u}_{i}, r_i)\}_{i=1}^{M}\right) 
\ \middle| \ 
\begin{aligned}
&\quad \mathbf{S}=(S_1,\dots,S_M) \in \mathbb{N}^M, \\
&\quad (\mathbf{u}_{1}, r_1),\ldots,(\mathbf{u}_{M}, r_M) \in \bigcup_{d=1}^{\infty} \mathbb{N}^d \times \mathbb{Z}_{\ge 0}, \\
& \left( e_0^{S_M-1} W_{\mathbf{u}_{M}} e_0^{r_M} \right) \cdots \left( e_0^{S_1-1} W_{\mathbf{u}_{1}} e_0^{r_1} \right)= W_{\mathbf{k}}~\text{in}~\{e_0,e_1\}^{\ast}
\end{aligned}
\right\}.
\]
\end{thm}

\begin{proof}
Theorem \ref{thm22} gives the expansion
\[
f(e_0, e_\infty) = f(0,0)+\sum_{\mathbf{S}} \operatorname{Coeff}_{\mathbf{S}}(f(e_0, e_\infty)) W_{\mathbf{S}}(e_0, e_1)+ (\text{terms ending in } e_0).
\]
Applying the substitution $e_1 \mapsto -e_0-e'_\infty$ transforms
$f(e_0, e_\infty)$ into $f(e_0, e'_\infty)$, and
$W_{\mathbf{S}}(e_0, e_1)$ into
\[
W_{\mathbf{S}}(e_0, -e_0-e'_\infty) = \prod_{i=M}^1 (e_0^{S_i-1}(-e_0-e'_\infty)) = (-1)^M \prod_{i=M}^1 (e_0^{S_i-1}(e_0+e'_\infty))
\]
where $M := \depth(\mathbf{S})$.
Using the expansion $e_0+e'_\infty = \sum \mu_{\mathbf{u}, r} W_{\mathbf{u}} e_0^r$,
we obtain
\begin{align*}
f(e_0, e'_\infty) = &f(0,0)+\sum_{\mathbf{S}} (-1)^M \operatorname{Coeff}_{\mathbf{S}}(f(e_0, e_\infty)) \prod_{i=M}^1 \left( \sum_{(\mathbf{u}, r) \in \mathcal{D}} \mu_{\mathbf{u}, r} e_0^{S_i-1} W_{\mathbf{u}} e_0^r \right)\\
&\quad +(\text{terms ending in } e_0)
\end{align*}
where $\mathcal{D} = \bigcup_{d=1}^\infty \mathbb{N}^{d} \times \mathbb{Z}_{\ge 0}$ is the set of all indices.
We expand the product $\prod_{i=M}^1$ as follows.
We select one term from the sum $\sum_{(\mathbf{u}, r) \in \mathcal{D}} \mu_{\mathbf{u}, r} e_0^{S_i-1} W_{\mathbf{u}} e_0^r$ for each $i=M,\cdots,1$, and represent this selection by the sequence
\[
\{(\mathbf{u}_{i}, r_i)\}_{i=1}^{M} \in \mathcal{D}^M.
\]
Then, the above product expands as a sum over $\mathcal{D}^M$:
\begin{align*}
&\sum_{\mathbf{S}} (-1)^M \operatorname{Coeff}_{\mathbf{S}}(f(e_0, e_\infty)) \prod_{i=M}^1 \left( \sum_{(\mathbf{u}, r) \in \mathcal{D}} \mu_{\mathbf{u}, r} e_0^{S_i-1} W_{\mathbf{u}} e_0^r \right)\\
&= \sum_{\mathbf{S}} (-1)^M \operatorname{Coeff}_{\mathbf{S}}(f(e_0, e_\infty)) \sum_{\{(\mathbf{u}_{i}, r_i)\}_{i=1}^{M} \in \mathcal{D}^M} \left( \prod_{i=1}^M \mu_{\mathbf{u}_{i}, r_i} \right) \cdot \Psi(\mathbf{S}, \{(\mathbf{u}_{i}, r_i)\}_{i=1}^{M})
\end{align*}
where
\[
\Psi(\mathbf{S}, \{(\mathbf{u}_{i}, r_i)\}_{i=1}^{M}) := \prod_{i=M}^1 \left( e_0^{S_i-1} W_{\mathbf{u}_{i}} e_0^{r_i} \right).
\]
The coefficient $\operatorname{Coeff}_{\mathbf{k}}(f(e_0, e'_\infty))$ is the sum of coefficients for all pairs $(\mathbf{S}, \{(\mathbf{u}_{i}, r_i)\}_{i=1}^{M})$ satisfying $\Psi(\mathbf{S}, \{(\mathbf{u}_{i}, r_i)\}_{i=1}^{M}) = W_{\mathbf{k}}$.
Since
\[
(\mathbf{S}, \{(\mathbf{u}_{i}, r_i)\}_{i=1}^{M}) \in \mathrm{Part}(\mathbf{k}, M) \iff \Psi(\mathbf{S}, \{(\mathbf{u}_{i}, r_i)\}_{i=1}^{M}) = W_{\mathbf{k}},
\]
rearranging the order of summation in the above equation yields:
\begin{align*}
&\operatorname{Coeff}_{\mathbf{k}}(f(e_0, e'_\infty)) \\
&= \sum_{M=1}^{\depth(\mathbf{k})} \sum_{\substack{\mathbf{S}, \{(\mathbf{u}_{i}, r_i)\}_{i=1}^{M} \\ \Psi(\mathbf{S}, \{(\mathbf{u}_{i}, r_i)\}) = W_{\mathbf{k}}}} (-1)^M \operatorname{Coeff}_{\mathbf{S}}(f(e_0, e_\infty)) \prod_{i=1}^M \mu_{\mathbf{u}_{i}, r_i}\\
&=\sum_{M=1}^{\depth(\mathbf{k})} \sum_{(\mathbf{S}, \{(\mathbf{u}_{i}, r_i)\}_{i=1}^{M}) \in \mathrm{Part}(\mathbf{k}, M)} (-1)^M \operatorname{Coeff}_{\mathbf{S}}(f(e_0, e_\infty)) \prod_{i=1}^M \mu_{\mathbf{u}_{i}, r_i}\\
&=\sum_{M=1}^{\depth(\mathbf{k})} \sum_{(\mathbf{S}, \{(\mathbf{u}_{i}, r_i)\}_{i=1}^{M}) \in \mathrm{Part}(\mathbf{k}, M)} (-1)^M \left(\sum_{\mathbf{J} \succeq \mathbf{S}} (-1)^{\depth(\mathbf{J})} \operatorname{Coeff}_{\mathbf{J}}(f(e_0, e_1))\right) \prod_{i=1}^M \mu_{\mathbf{u}_{i}, r_i}\\
&\quad (\text{by Theorem }\ref{thm22})\\
&=\sum_{M=1}^{\depth(\mathbf{k})} (-1)^M \left(\sum_{\substack{\left(\mathbf{S}, \{(\mathbf{u}_{i}, r_i)\}_{i=1}^{M}\right) \in \mathrm{Part}(\mathbf{k}, M), \\ \exists i, (\mathbf{u}_{i}, r_i) \neq ((1), 0)}} \left(\sum_{\mathbf{J} \succeq \mathbf{S}} (-1)^{\depth(\mathbf{J})} \operatorname{Coeff}_{\mathbf{J}}(f(e_0, e_1))\right) \prod_{i=1}^M \mu_{\mathbf{u}_{i}, r_i}\right)\\
&\quad +\sum_{\mathbf{J} \succeq \mathbf{k}} (-1)^{\depth(\mathbf{J})} \operatorname{Coeff}_{\mathbf{J}}(f(e_0, e_1)).
\end{align*}
The last equality follows by isolating the term corresponding to $(\mathbf{u}_i, r_i) = ((1), 0)$ for all $i=1,\cdots,M$ which implies $W_{\mathbf{S}}=W_{\mathbf{k}}$, $M = \depth(\mathbf{k})$ and $\mu_{(1),0}=-1$.
This completes the proof.
\end{proof}

\begin{cor}[Coefficient of $W_{n}$ in $f(e_0, e'_\infty)$]\label{cor:single_index}
Let $n$ be a positive integer. For any $f(e_0, e_1) \in K\langle\langle e_0, e_1 \rangle\rangle$, the coefficient $\operatorname{Coeff}_{n}(f(e_0, e'_\infty))$ of $W_{n} = e_0^{n-1}e_1$ in the expansion of $f(e_0, e'_\infty)$ is expressed as follows using the coefficients of $f(e_0, e_1)$:
\begin{align*}
\operatorname{Coeff}_{n}(f(e_0, e'_\infty)) 
=\sum_{u=1}^{n} \frac{(-1)^{u+1} B_{u-1}}{(u-1)!} \left( \sum_{\substack{\mathbf{J} \in \mathbb{N}^{m},~m \in \mathbb{N}, \\ {\rm wt}(\mathbf{J})=n-u+1}} (-1)^{\depth(\mathbf{J})} \operatorname{Coeff}_{\mathbf{J}}(f(e_0, e_1)) \right).
\end{align*}
Here, $B_k$ is the Bernoulli number defined with $B_1 = -1/2$.
\end{cor}

\begin{proof}
Applying Theorem \ref{thm33} with $\mathbf{k}=(n)$, we have $\depth(\mathbf{k})=1$, so $M=1$. The condition for $\mathrm{Part}((n), 1)$ implies $e_0^{S_1-1} W_{\mathbf{u}_{1}} e_0^{r_1} = e_0^{n-1}e_1$. This forces $\depth(\mathbf{u}_1)=1$, so $\mathbf{u}_1=(u)$.
Comparing the degrees and the order of terms yields $r_1=0$ and $S_1+u-1=n$.
Substituting $\mu_{(u),0} = \frac{(-1)^u B_{u-1}}{(u-1)!}$ (Proposition \ref{prop:integral_rep_corrected2}) into the formula completes the proof.
\end{proof}

\section{Review of multiple polylogarithms}

In this section,
we review complex multiple polylogarithms and $\ell$-adic  Galois multiple polylogarithms.

\subsection{Complex multiple polylogarithms}\label{cmpoly}
In this subsection,
we recall the definition of the complex multiple polylogarithm \cite{D90}, \cite{F04} and \cite{F14}.

For a (possibly, tangential base) point $z$ on $\mathbb{P}^1 \backslash \{0,1,\infty\}$,
we denote by
\[
\pi_1^{\rm top}\left(\mathbb{P}^1(\mathbb{C}) \backslash \{0,1,\infty\}; \overrightarrow{01}, z\right)
\]
the set of  homotopy classes of piecewise smooth topological paths on
$\mathbb{P}^1(\mathbb{C}) \backslash \{0,1,\infty\}$ from the tangential base point $\overrightarrow{01}$ to $z$,
and
\[
\pi_1^{\rm top}\left(\mathbb{P}^1(\mathbb{C}) \backslash \{0,1,\infty\}, \overrightarrow{01}\right):=\pi_1^{\rm top}\left(\mathbb{P}^1(\mathbb{C}) \backslash \{0,1,\infty\}; \overrightarrow{01}, \overrightarrow{01}\right)
\]
for the topological fundamental group of $\mathbb{P}^1(\mathbb{C}) \backslash \{0,1,\infty\}$ at the base point $\overrightarrow{01}$
with respect to the path composition
$
\gamma_1 \cdot \gamma_2:=\gamma_1 \gamma_2,
$
i.e., paths are composed from left to right.
Let 
\[
l_0,
l_1,
l_\infty \in \pi_1^{\rm top}\left(\mathbb{P}^1(\mathbb{C}) \backslash \{0,1,\infty\}, \overrightarrow{01}\right)
\]
be homotopy classes of smooth loops
circling counterclockwise around
$0$, $1$, and $\infty$ respectively,
as shown in Figure \ref{paths}.
Then,
$\pi_1^{\rm top}\left(\mathbb{P}^1(\mathbb{C}) \backslash \{0,1,\infty\}, \overrightarrow{01}\right)$
is the free group of rank $2$ with the generating system $\{l_0, l_1\}$ subject to the relation $l_0 l_1 l_\infty = 1$.

\begin{center}
\begin{figure}[hbtp]
\caption{Topological loops on ${\mathbb P}^1({\mathbb C})\backslash 
\{0,1,\infty\}$}
\label{paths}
\begin{tikzpicture} \label{picture2}

\draw (0.1,0) -- (0,0.1);
\draw (0.1,0) -- (0.2,0.1);

\draw (7.3,0) -- (7.4,-0.1);
\draw (7.3,0) -- (7.2,-0.1);

\draw (11.3,0) -- (11.4,-0.1);
\draw (11.3,0) -- (11.2,-0.1);


\draw (2.1,0) to [out=0,in=270] (2.8,1);
\draw (2.8,1) to [out=90,in=-10] (2,1.9);
\draw (2,1.9) to [out=170,in=90] (0.1,0);
\draw (0.1,0) to [out=270,in=190] (2,-1.9);
\draw (2,-1.9) to [out=10,in=270] (2.8,-0.7);

\draw (2.8,-0.7) to [out=90,in=360] (2.1,0);
\draw (2,0) to [out=0,in=210] (4,0.4);
\draw (4,0.4) to [out=30,in=180] (6,1.4);
\draw (6,1.4) to [out=0,in=90] (7.3,0);
\draw (7.3,0) to [out=270,in=0] (6,-1.4);
\draw (4,-0.4) to [out=330,in=180] (6,-1.4);
\draw (4,-0.4) to [out=150,in=0] (2,0);

\draw (2.1,0) to [out=0,in=225] (4.5,1.5);
\draw (4.5,1.5) to [out=45,in=180] (8,3.3);
\draw (8,3.3) to [out=0,in=90] (11.3,0);
\draw (11.3,0) to [out=270,in=0] (10,-1.4);
\draw (8.6,0) to [out=-90,in=180] (10,-1.4);
\draw (8.6,0) to [out=90,in=0] (6.5,2);
\draw (2.2,0) to [out=0,in=180] (6.5,2);

\node at (-0.2,0) {$l_0$};
\node at (1.8,-0.4) {$0$};
\node at (6.2,-0.4) {$1$};
\node at (10.2,-0.4) {$\infty$};
\node at (7.7,0) {$l_1$};
\node at (11.7,0) {$l_\infty$};
\node at (2,0) {${\bullet}$};
\node at (6,0) {${\bullet}$};
\node at (10,0) {${\bullet}$};
\node at (10,3.6) {~};

\draw[dotted](-1,0)--(12.5,0);

\fill[white](2,0)circle(0.07);
\fill[white](6,0)circle(0.07);
\fill[white](10,0)circle(0.07);

\node at (2,0) {${\circ}$};
\node at (6,0) {${\circ}$};
\node at (10,0) {${\circ}$};
\end{tikzpicture}
\end{figure}
\end{center}

\begin{dfn}[Complex multiple polylogarithms $Li_{\mathbf{k}}$]
Let $z$ be a $\mathbb{C}$-rational (possibly tangential) base point on $\mathbb{P}^1 \backslash \{0,1,\infty\}$.
For a multi-index $\mathbf{k}=(k_1, \dots, k_d) \in \mathbb{N}^d$ and a path $\gamma \in \pi_1^{\rm top}\left(\mathbb{P}^1(\mathbb{C}) \backslash \{0,1,\infty\}; \overrightarrow{01}, z\right)$, we define the complex multiple polylogarithm $Li_{\mathbf{k}}(z;\gamma)$ by
the following iterated integral along $\gamma$:
\begin{align}\label{cpolydef}
& {Li}_{\mathbf{k}}\left(z;\gamma\right):=
\begin{cases}
\displaystyle \int_{\gamma} \frac{1}{t}Li_{k_1,\dots,k_d-1}\left(t;\gamma_{}\right)dt & k_d\neq 1, \\
\displaystyle \int_{\gamma} \frac{1}{1-t}Li_{k_1,\dots,k_{d-1}}\left(t;\gamma_{}\right)dt & k_d= 1, \\
\end{cases} \\ \label{cpolydef2}
& Li_1\left(z;\gamma\right):=-\log(1-z;\gamma)=\int_{\gamma}\frac{dt}{1-t}.
\end{align}
\end{dfn}

\begin{dfn}[De Rham variables $e_0^{\rm dR}, e_1^{\rm dR},$ and $e_\infty^{\rm dR}$] \label{dfn_variables}
Let 
$$V := \Omega^1_{\rm log}\left( \mathbb{P}^1(\mathbb{C}) \backslash \{0,1,\infty\} \right)=\mathbb{C}\dfrac{dz}{z} \oplus \mathbb{C}\dfrac{dz}{z-1}$$ 
be the space of meromorphic 1-forms with logarithmic singularities on $( \mathbb{P}^1(\mathbb{C}), \{0,1,\infty\})$.
Recall that for a logarithmic 1-form $\eta \in V$, its residue at a pole $p \in \{0, 1, \infty\}$ is defined as the coefficient $a_p \in \mathbb{C}$ in the local expansion $\eta = a_p \frac{dt_p}{t_p} + \text{(holomorphic)}$ around $p$, where $t_p$ is a local coordinate at $p$, 
i.e., $t_p = z-p$ for $p \in \{0, 1\}$ and $t_\infty = 1/z$. 
This naturally induces a $\mathbb{C}$-linear residue map
$$\text{\rm Res}_{p}: V \to \mathbb{C}, \eta \mapsto a_p$$
which canonically extends to $\text{\rm Res}_{p}: V \otimes_{\mathbb{C}} W \to W$ for any $\mathbb{C}$-vector space $W$.
Let $e_0^{\rm dR}$ and $e_1^{\rm dR}$ be the dual basis of $V$ corresponding to $\frac{dz}{z}$ and $\frac{dz}{z-1}$, respectively:
\begin{align}\label{Cvariable}
e_0^{\rm dR}:=\left(\dfrac{dz}{z}\right)^{\ast},~
e_1^{\rm dR}:=\left(\dfrac{dz}{z-1}\right)^{\ast}~\in V^{\ast}.
\end{align}
Consider the canonical 1-form
\begin{align}\label{Cvariable_omega}
\omega := \frac{dz}{z} e_0^{\rm dR} + \frac{dz}{z-1} e_1^{\rm dR} \in V \otimes_{\mathbb{C}} V^\ast.
\end{align}
By construction, the residues of $\omega$ at $z=0$ and $z=1$ coincide exactly with the dual of differential forms:
\begin{align}\label{residue_01}
\text{\rm Res}_{0}(\omega) = e_0^{\rm dR}, \quad \text{\rm Res}_{1}(\omega) = e_1^{\rm dR}.
\end{align}
We define the variable $e_\infty^{\rm dR}$ as the residue of $\omega$ at $z=\infty$:
\begin{align}\label{residue_def}
e_\infty^{\rm dR} := \text{\rm Res}_{\infty}(\omega).
\end{align}
\end{dfn}

\begin{rem} \label{rem_residue_sum}
Under the natural isomorphism
\[
\pi_1^{\rm top}(\mathbb{P}^1(\mathbb{C}) \backslash \{0,1,\infty\}, \overrightarrow{01})^{\rm ab} \otimes \mathbb{C}
\simeq
\Omega^1_{\rm log}\left( \mathbb{P}^1(\mathbb{C}) \backslash \{0,1,\infty\} \right)^{\ast},
\]
the de Rham variables $e_0^{\rm dR}, e_1^{\rm dR},$ and $e_\infty^{\rm dR}$ correspond to the generators of the abelianized topological fundamental group:
\begin{align} \label{identification}
\dfrac{\bar{l}_0}{2\pi \sqrt{-1}}=e_0^{\rm dR},\quad
\dfrac{\bar{l}_1}{2\pi \sqrt{-1}}=e_1^{\rm dR},\quad
\dfrac{\bar{l}_\infty}{2\pi \sqrt{-1}}=e_\infty^{\rm dR},
\end{align}
where $\bar{l}_i$ denotes the class of $l_i$ in the abelianization.
On the other hand, the Cauchy residue theorem on $\mathbb{P}^1(\mathbb{C})$ implies that the sum of the residues of the 1-form $\omega$ vanishes:
\[
\text{\rm Res}_{0}(\omega) + \text{\rm Res}_{1}(\omega) + \text{\rm Res}_{\infty}(\omega) = 0.
\]
Together with (\ref{residue_01}) and (\ref{residue_def}) this yields the linear relation
\begin{align}\label{Cvariablerel}
e_\infty^{\rm dR}=-e_1^{\rm dR}-e_0^{\rm dR}.
\end{align}
In the abelianization of $\pi_1^{\rm top}\left(\mathbb{P}^1(\mathbb{C}) \backslash \{0,1,\infty\}, \overrightarrow{01}\right)=\left<l_0,
l_1,
l_{\infty} \mid l_0 l_1 l_{\infty}=1 \right>$, the group-theoretic relation $l_0 l_1 l_\infty = 1$ becomes
$$\bar{l}_0 + \bar{l}_1 + \bar{l}_\infty = 0,$$
which is perfectly consistent with (\ref{Cvariablerel}) via the identification (\ref{identification}).
\end{rem}

\begin{dfn}[Generating series $\Lambda_{\overrightarrow{01}}^{z, \gamma}$]
Let $\mathbb{C} \langle \langle e_0^{\rm dR}, e_1^{\rm dR} \rangle \rangle$ be the algebra of non-commutative formal power series in two variables over $\mathbb{C}$.
For a path $\gamma \in \pi_1^{\rm top}\left(\mathbb{P}^1(\mathbb{C}) \backslash \{0,1,\infty\}; \overrightarrow{01}, z\right)$, we define the formal power series $\Lambda_{\overrightarrow{01}}^{z, \gamma}\left(e_0^{\rm dR},e_1^{\rm dR}\right)$ by
\begin{align}\label{Lambda}
\Lambda_{\overrightarrow{01}}^{z, \gamma}\left(e_0^{\rm dR},e_1^{\rm dR}\right):=1+\sum_{i=1}^{\infty}~\int_{ \overrightarrow{01},\gamma}^{z}~\underbrace{\omega \circ \cdots \circ \omega}_{i~\text{times}}~\in \mathbb{C} \langle \langle e_0^{\rm dR}, e_1^{\rm dR} \rangle \rangle^{\times}
\end{align}
where $\mathbb{C} \langle \langle e_0^{\rm dR}, e_1^{\rm dR} \rangle \rangle$ is identified with the completed tensor algebra of $V^\ast$, and $\omega \circ \cdots \circ \omega$ denotes the product of the 1-forms induced by the tensor algebra.
Here, through the natural inclusion $V^{\ast} \hookrightarrow \mathbb{C}\langle\langle e_0^{\rm dR}, e_1^{\rm dR} \rangle\rangle$, we regard the canonical 1-form $\omega$ as an element of $\Omega^1_{\rm log}\left( \mathbb{P}^1(\mathbb{C}) \backslash \{0,1,\infty\} \right) \otimes_{\mathbb{C}} \mathbb{C} \langle \langle e_0^{\rm dR}, e_1^{\rm dR} \rangle \rangle$ to form the iterated products $\omega \circ \cdots \circ \omega$.
See \cite{W97} for details on (\ref{Lambda}).
\end{dfn}

\begin{dfn}[Complex KZ solution $G_{\overrightarrow{01}}^{z, \gamma}$]
For a word $w=w_1 \cdots w_n$ with $w_i \in \{e_0^{\rm dR},e_1^{\rm dR}\}$,
we define ${w}^{\rm op}:=w_n \cdots w_1$.
For $\Lambda=1+ \sum_{w \in \{e_0^{\rm dR},e_1^{\rm dR}\}^{\ast} \backslash \{1\}} {\rm Coeff}_{w}(\Lambda)~w  \in \mathbb{C} \langle \langle e_0^{\rm dR},e_1^{\rm dR} \rangle \rangle$,
we define the opposite of $\Lambda$ by
\[
{\Lambda}^{\rm op}:=1+\sum_{w \in \{e_0^{\rm dR},e_1^{\rm dR}\}^{\ast} \backslash \{1\}} {\rm Coeff}_{w}(\Lambda)~{w}^{\rm op}.
\]
We define
\begin{align}\label{KZsol}
G_{\overrightarrow{01}}^{z, \gamma}\left(e_0^{\rm dR},e_1^{\rm dR}\right):=\left(\Lambda_{\overrightarrow{01}}^{z, \gamma}\left(e_0^{\rm dR},e_1^{\rm dR}\right)\right)^{\rm op}.
\end{align}
\end{dfn}

\begin{rem}
The formal Knizhnik-Zamolodchikov (KZ) equation on $\mathbb{P}^1(\mathbb{C}) \backslash \{0,1,\infty\}$ is the differential equation
\[
\dfrac{d}{d z}G\left(e_0^{\rm dR},e_1^{\rm dR}\right)(z)=\left( \dfrac{e_0^{\rm dR}}{z}+\dfrac{e_1^{\rm dR}}{z-1} \right)G\left(e_0^{\rm dR},e_1^{\rm dR}\right)(z),
\]
where
$G\left(e_0^{\rm dR},e_1^{\rm dR}\right)(z)$ is an analytic function with values in $\mathbb{C} \langle \langle e_0^{\rm dR}, e_1^{\rm dR} \rangle \rangle$.
The formal power series $G_{\overrightarrow{01}}^{z, \gamma}\left(e_0^{\rm dR},e_1^{\rm dR}\right)$ is a fundamental solution of this equation.
\end{rem}

\begin{prop}
For a $\mathbb{C}$-rational (possibly tangential) base point $z$ on $\mathbb{P}^1 \backslash \{0,1,\infty\}$ and a path $\gamma \in \pi_1^{\rm top}\left(\mathbb{P}^1(\mathbb{C}) \backslash \{0,1,\infty\}; \overrightarrow{01}, z\right)$,
the following holds
\begin{align}
G_{\overrightarrow{01}}^{z, \gamma}\left(e_0^{\rm dR},e_1^{\rm dR}\right)=1+\sum_{\mathbf{k}=(k_1, \dots, k_d)} (-1)^{{\rm dp}(\mathbf{k})} Li_{\mathbf{k}}(z;\gamma)~{\underbrace{e_0^{\rm dR} \ldots e_0^{\rm dR}}_{k_d-1~\text{times}}}e_1^{\rm dR} \cdots {\underbrace{e_0^{\rm dR} \ldots e_0^{\rm dR}}_{k_1-1~\text{times}}}e_1^{\rm dR}+\cdots. \notag
\end{align}
In other words, for $W_\mathbf{k}:= {\underbrace{e_0^{\rm dR} \ldots e_0^{\rm dR}}_{k_d-1~\text{times}}}e_1^{\rm dR} \cdots {\underbrace{e_0^{\rm dR} \ldots e_0^{\rm dR}}_{k_1-1~\text{times}}}e_1^{\rm dR}$, we obtain
\begin{align}\label{formcpoly}
Li_{\mathbf{k}}(z;\gamma)=(-1)^{{\rm dp}(\mathbf{k})} {\rm Coeff}_{W_\mathbf{k}}\left(G_{\overrightarrow{01}}^{z, \gamma}\left(e_0^{\rm dR},e_1^{\rm dR}\right)\right).
\end{align}
\end{prop}

\begin{proof}
By the definition (\ref{cpolydef}) and (\ref{Lambda}),
it is verified that
the expansion of (\ref{Lambda}) has the form
\begin{align}\label{G0000}
\Lambda_{\overrightarrow{01}}^{z, \gamma}\left(e_0^{\rm dR},e_1^{\rm dR}\right)=&~1+\sum_{\mathbf{k}=(k_1, \dots, k_d)} (-1)^{{\rm dp}(\mathbf{k})} Li_{\mathbf{k}}(z;\gamma)~e_1^{\rm dR}{\underbrace{e_0^{\rm dR} \ldots e_0^{\rm dR}}_{k_1-1~\text{times}}} \cdots e_1^{\rm dR}{\underbrace{e_0^{\rm dR} \ldots e_0^{\rm dR}}_{k_d-1~\text{times}}}+\cdots. 
\end{align}
The assertion follows from (\ref{KZsol}) and (\ref{G0000}).
\end{proof}

\begin{rem}\label{rem:shuffle_cpoly_Li1}
Since the complex KZ solution $G_{\overrightarrow{01}}^{z, \gamma}(e_0^{\rm dR}, e_1^{\rm dR})$ is a group-like element in the complete Hopf algebra $\mathbb{C} \langle \langle e_0^{\rm dR}, e_1^{\rm dR} \rangle \rangle$, its coefficients satisfy the shuffle relation (cf. \cite[3.2.1]{Z16}).
This shuffle relation forces the coefficient of $(e_1^{\rm dR})^n$ to be determined by the $n$-th power of the coefficient of $e_1^{\rm dR}$:$${\rm Coeff}_{(e_1^{\rm dR})^n}\left(G_{\overrightarrow{01}}^{z, \gamma}(e_0^{\rm dR}, e_1^{\rm dR})\right) = \frac{1}{n!} \left( {\rm Coeff}_{e_1^{\rm dR}}\left(G_{\overrightarrow{01}}^{z, \gamma}(e_0^{\rm dR}, e_1^{\rm dR})\right) \right)^n.$$
By the formula (\ref{formcpoly}),
we have ${\rm Coeff}_{e_1^{\rm dR}}\left(G_{\overrightarrow{01}}^{z, \gamma}(e_0^{\rm dR}, e_1^{\rm dR})\right) = -Li_1(z;\gamma)$. Thus, for the index $\mathbf{k} = (1, 1, \dots, 1)$ of depth $n$, we obtain
$$\begin{aligned}\label{shuffle_cpoly_Li1}
Li_{1,1,\dots,1}(z;\gamma) 
&= (-1)^n {\rm Coeff}_{(e_1^{\rm dR})^n}\left(G_{\overrightarrow{01}}^{z, \gamma}(e_0^{\rm dR}, e_1^{\rm dR})\right) \\
&= \frac{1}{n!} \left( Li_1(z;\gamma) \right)^n.
\end{aligned}$$
As recalled in (\ref{cpolydef2}), the complex multiple polylogarithm of depth $1$ coincides with the negative logarithm:
$Li_1(z;\gamma) = -\log(1-z;\gamma).$
Therefore, it holds that
\begin{align}
Li_{1,1,\dots,1}(z;\gamma) = \frac{1}{n!} \left( -\log(1-z;\gamma) \right)^n.
\end{align}
\end{rem}

\subsection{$\ell$-adic Galois multiple polylogarithms}\label{lmpoly}
In this subsection,
we recall the definition of the $\ell$-adic Galois multiple polylogarithm.

Let $\ell$ be a prime number and
$K$ a subfield of $\mathbb{C}$ with the algebraic closure $\overline{K} \subset \mathbb{C}$.
Suppose that $z$ is a $K$-rational (possibly tangential) base point on $\mathbb{P}^1 \backslash \{0,1,\infty\}$.

We denote by
\[
\pi_1^\ellet\left(\mathbb{P}^1_{\overline{K}} \backslash \{0,1,\infty\};\overrightarrow{01},{z}\right)
\]
the pro-$\ell$ set of pro-$\ell$ \'etale paths
 on $\mathbb{P}^1_{\overline{K}} \backslash \{0,1,\infty\}$ from the $K$-rational tangential base point $\overrightarrow{01}$ to ${z}$,
and
\[
\pi_1^\ellet\left(\mathbb{P}^1_{\overline{K}} \backslash \{0,1,\infty\},\overrightarrow{01}\right):=\pi_1^\ellet\left(\mathbb{P}^1_{\overline{K}} \backslash \{0,1,\infty\};\overrightarrow{01},\overrightarrow{01}\right)
\]
the pro-$\ell$ \'etale fundamental group of $\mathbb{P}^1_{\overline{K}} \backslash \{0,1,\infty\}$ with the base point $\overrightarrow{01}$.
By the comparison map
\[
\pi_1^{\rm top}\left(\mathbb{P}^1(\mathbb{C}) \backslash \{0,1,\infty\}, \overrightarrow{01}\right) \hookrightarrow \pi_1^\ellet\left(\mathbb{P}^1_{\overline{K}} \backslash \{0,1,\infty\},\overrightarrow{01}\right)
\]
induced by the inclusion $\overline{K} \hookrightarrow \mathbb{C}$,
we regard topological loops $l_0,l_1$ on $\mathbb{P}^1(\mathbb{C}) \backslash \{0,1,\infty\}$ as pro-$\ell$ \'etale loops on $\mathbb{P}^1_{\overline{K}} \backslash \{0,1,\infty\}$.
Then $\pi_1^\ellet\left(\mathbb{P}^1_{\overline{K}} \backslash \{0,1,\infty\},\overrightarrow{01}\right)$ is the free pro-$\ell$ group of rank $2$ with the topological generating system $\{l_0,l_1\}$.
\[
\pi_1^\ellet\left(\mathbb{P}^1_{\overline{K}} \backslash \{0,1,\infty\},\overrightarrow{01}\right)=\overline{\left<l_0,
l_1\right>}.
\]

We focus on the natural action of $G_K$ on $\pi_1^\ellet \left(\mathbb{P}^1_{\overline{K}} \backslash \{0,1,\infty\};\overrightarrow{01},{z}\right)$.
Since $z$ is $K$-rational,
this Galois action is well-defined.

\begin{dfn}[Pro-$\ell$ Galois 1-cocycle ${\mathfrak f}^{z,\gamma}$]
For each $\gamma \in \pi_1^\ellet\left(\mathbb{P}^1_{\overline{K}} \backslash \{0,1,\infty\};\overrightarrow{01},{z}\right)$,
we define a non-commutative Galois 1-cocycle
\[
{\mathfrak f}^{z,\gamma} \in Z^1\left(G_K, \pi_1^\ellet\left(\mathbb{P}^1_{\overline{K}} \backslash \{0,1,\infty\},\overrightarrow{01}\right)\right)
\]
by
\begin{align} \label{f}
{\mathfrak f}^{z,\gamma}: G_K \to  \pi_1^\ellet\left(\mathbb{P}^1_{\overline{K}} \backslash \{0,1,\infty\},\overrightarrow{01}\right), \quad \sigma \mapsto {\mathfrak f}^{z,\gamma}_{\sigma}:=\gamma \cdot \sigma(\gamma)^{-1}.
\end{align}
\end{dfn}

\begin{dfn}[Betti variables $e_0^{\rm B}, e_1^{\rm B},$ and $e_\infty^{\rm B}$]
We set
\begin{align}\label{Bvari}
e_0^{\mathrm{B}}:={\bf log}(l_0)\left(=\sum_{n=1}^{\infty} \frac{(-1)^{n-1}}{n}(l_0-1)^n\right),\quad e_1^{\mathrm{B}}:={\bf log}(l_1),\quad e_\infty^{\mathrm{B}}:={\bf log}(l_\infty)
\end{align}
in the complete group ring of $\pi_1^\ellet\left(\mathbb{P}^1_{\overline{K}} \backslash \{0,1,\infty\},\overrightarrow{01}\right)$ over $\mathbb{Q}_{\ell}$.
\end{dfn}

\begin{rem}
In contrast to the complex case (\ref{Cvariablerel}),
the following relation holds.
\begin{align}\label{lvariablerel}
e_\infty^{\mathrm{B}}={\bf log}\left({\bf exp}\left(-e_1^{\rm B}\right){\bf exp}\left(-e_0^{\rm B}\right)\right)=-e_1^{\rm B}-e_0^{\rm B}+\text{$($higher-order terms$)$}.
\end{align}
\end{rem}

\begin{dfn}[$\ell$-adic Galois 1-cocycle ${\mathfrak f}^{z,\gamma}$]
Let $\mathbb{Q}_{\ell} \langle \langle e_0^{\rm B}, e_1^{\rm B} \rangle \rangle$ be the algebra of non-commutative formal power series in two variables over $\mathbb{Q}_{\ell}$.
Consider the $\ell$-adic Magnus embedding
\begin{align}\label{lmag}
\pi_1^\ellet\left(\mathbb{P}^1_{\overline{K}} \backslash \{0,1,\infty\},\overrightarrow{01}\right) \hookrightarrow \mathbb{Q}_{\ell}\left[\left[\pi_1^\ellet\left(\mathbb{P}^1_{\overline{K}} \backslash \{0,1,\infty\},\overrightarrow{01}\right)\right]\right] \simeq \mathbb{Q}_{\ell} \langle \langle e_0^{\mathrm{B}},
e_1^{\mathrm{B}} \rangle \rangle
\end{align}
defined by
\[
l_0={\bf exp}(e_0^{\mathrm{B}})\left(=\sum_{n=0}^{\infty}\frac{1}{n!}{(e_0^{\mathrm{B}})}^n\right),\quad
l_1={\bf exp}(e_1^{\mathrm{B}}).
\]
For $\gamma \in \pi_1^\ellet\left(\mathbb{P}^1_{\overline{K}} \backslash \{0,1,\infty\};\overrightarrow{01},{z}\right)$, 
we define a non-commutative Galois 1-cocycle
\[
{\mathfrak f}^{z,\gamma} \in Z^1\left(G_K, \mathbb{Q}_{\ell} \langle \langle e_0^{\mathrm{B}},
e_1^{\mathrm{B}} \rangle \rangle^{\times}\right)
\]
by
\begin{align} \label{f2}
{\mathfrak f}^{z,\gamma} : G_K \to \pi_1^\ellet\left(\mathbb{P}^1_{\overline{K}} \backslash \{0,1,\infty\},\overrightarrow{01}\right) \hookrightarrow \mathbb{Q}_{\ell} \langle \langle e_0^{\mathrm{B}},
e_1^{\mathrm{B}} \rangle \rangle^{\times}, \quad \sigma \mapsto {\mathfrak f}^{z,\gamma}_{\sigma} \mapsto {\mathfrak f}^{z,\gamma}_{\sigma}(e_0^{\mathrm{B}},
e_1^{\mathrm{B}}).
\end{align}
By abuse of notation, we use the same symbol ${\mathfrak f}^{z,\gamma}$ for both the pro-$\ell$ Galois 1-cocycle and its image under the Magnus embedding.
\end{dfn}

\begin{rem}
The formal power series ${\mathfrak f}^{z,\gamma}_{\sigma}(e_0^{\mathrm{B}},
e_1^{\mathrm{B}})$ is the $\ell$-adic Galois analogue of the complex KZ fundamental solution $G_{\overrightarrow{01}}^{z, \gamma}\left(e_0^{\rm dR},e_1^{\rm dR}\right)$ in (\ref{KZsol}).
Since
${\mathfrak f}^{z,\gamma}_{\sigma}(e_0^{\mathrm{B}},
e_1^{\mathrm{B}})$ is group-like in $\mathbb{Q}_{\ell} \langle \langle e_0^{\mathrm{B}},
e_1^{\mathrm{B}} \rangle \rangle$,
there exists
a family of $\ell$-adic numbers
$\left\{{\rm Coeff}_{w}\left({\mathfrak f}^{z,\gamma}_{\sigma}(e_0^{\mathrm{B}},
e_1^{\mathrm{B}})\right)\right\}_{w \in \{e_0^{\rm B},e_1^{\rm B}\}^{\ast}}$
such that
the expansion of ${\mathfrak f}^{z,\gamma}_{\sigma}(e_0^{\mathrm{B}},
e_1^{\mathrm{B}})$ takes the form
\begin{equation}
{\mathfrak f}^{z,\gamma}_{\sigma}(e_0^{\mathrm{B}},
e_1^{\mathrm{B}})=1+\sum_{w \in \{e_0^{\rm B},e_1^{\rm B}\}^{\ast} \backslash \{1\}} {\rm Coeff}_{w}\left({\mathfrak f}^{z,\gamma}_{\sigma}(e_0^{\mathrm{B}},
e_1^{\mathrm{B}})\right)~w,
\end{equation}
where $\{e_0^{\mathrm{B}},e_1^{\mathrm{B}}\}^{\ast}$ is the non-commutative free monoid generated by the non-commuting indeterminates $e_0^{\mathrm{B}}$ and $e_1^{\mathrm{B}}$.
\end{rem}

\begin{dfn}[$\ell$-adic Galois multiple polylogarithm]\label{lpolydef}
For $\mathbf{k} \in \mathbb{N}^d$, $\sigma \in G_K$ and $\gamma \in \pi_1^\ellet\left(\mathbb{P}^1_{\overline{K}} \backslash \{0,1,\infty\};\overrightarrow{01},{z}\right)$, we define the $\ell$-adic Galois multiple polylogarithm
by 
\begin{align} \label{lpoly1}
{Li}^\ell_{\mathbf{k}}(z;\gamma,\sigma)
&:=(-1)^{d} \cdot {\rm Coeff}_{W_\mathbf{k}}\left({\mathfrak f}^{z,\gamma}_{\sigma}(e_0^{\mathrm{B}}, e_1^{\mathrm{B}})\right)
\end{align}
where $W_\mathbf{k}:= {\underbrace{e_0^{\rm B} \ldots e_0^{\rm B}}_{k_d-1~\text{times}}}e_1^{\rm B} \cdots {\underbrace{e_0^{\rm B} \ldots e_0^{\rm B}}_{k_1-1~\text{times}}}e_1^{\rm B}$.
\end{dfn}

\begin{rem}
The $\ell$-adic number (\ref{lpoly1}) depends on the choices of $\sigma$ and $\gamma$.
It is the higher-depth generalization of
the $\ell$-adic Galois polylogarithm introduced and explored in \cite{W0}-\cite{W5}, \cite{NW99}, \cite{NW12}, \cite{NSW17a}, \cite{NSW17b}, \cite{NW20a} and \cite{NW20b}.
The Oi-Ueno type functional equation for (\ref{lpoly1}) and its generalization were established
in \cite{NS25} and \cite{S24}.
The paper \cite{F07} presents a construction of the complex multiple polylogarithm using the complex Magnus embedding based on a Tannakian formalism; this construction makes the analogy between the complex multiple polylogarithm and the $\ell$-adic Galois multiple polylogarithm even clearer.
\end{rem}

\begin{rem}\label{rem:shuffle_Li1}
Since the formal power series ${\mathfrak f}^{z,\gamma}_{\sigma}(e_0^{\mathrm{B}}, e_1^{\mathrm{B}})$ is a group-like element in the complete Hopf algebra $\mathbb{Q}_{\ell} \langle \langle e_0^{\mathrm{B}}, e_1^{\mathrm{B}} \rangle \rangle$, its coefficients satisfy the shuffle relation (cf. \cite[3.2.1]{Z16}). 
This shuffle relation forces the coefficient of $(e_1^{\mathrm{B}})^n
$ to be determined by the $n$-th power of the coefficient of $e_1^{\mathrm{B}}$:
\[
{\rm Coeff}_{(e_1^{\mathrm{B}})^n}\left({\mathfrak f}^{z,\gamma}_{\sigma}(e_0^{\mathrm{B}}, e_1^{\mathrm{B}})\right) = \frac{1}{n!} \left( {\rm Coeff}_{e_1^{\mathrm{B}}}\left({\mathfrak f}^{z,\gamma}_{\sigma}(e_0^{\mathrm{B}}, e_1^{\mathrm{B}})\right) \right)^n.
\]
By Definition \ref{lpolydef}, we have ${\rm Coeff}_{e_1^{\mathrm{B}}}\left({\mathfrak f}^{z,\gamma}_{\sigma}(e_0^{\mathrm{B}}, e_1^{\mathrm{B}})\right) = -Li^\ell_1(z;\gamma,\sigma)$. Thus, for the index $\mathbf{k} = (1, 1, \dots, 1)$ of depth $n$, we obtain
\begin{align*}\label{shuffle_Li1}
Li^\ell_{1,1,\dots,1}(z;\gamma,\sigma) 
&= (-1)^n {\rm Coeff}_{(e_1^{\mathrm{B}})^n}\left({\mathfrak f}^{z,\gamma}_{\sigma}(e_0^{\mathrm{B}}, e_1^{\mathrm{B}})\right) \notag \\
&= \frac{1}{n!} \left( Li^\ell_1(z;\gamma,\sigma) \right)^n. 
\end{align*}
It is a fundamental fact that $Li^\ell_1(z;\gamma,\sigma)$ coincides with the $\ell$-adic Kummer 1-cocycle evaluated at $1-z$ defined by $\sigma\left((1-z)^{1/\ell^n}\right)=\zeta_{\ell^n}^{\rho_{{1-z},\gamma}(\sigma)} \cdot (1-z)^{1/\ell^n}$
with respect to $\{(1-z)^{1/{\ell^n}}\}_{n \in \mathbb{N}}$ along $\gamma \in \pi_1^{\rm top}\left(\mathbb{P}^1(\mathbb{C}) \backslash \{0,1,\infty\}; \overrightarrow{01}, z\right)$
 (cf. \cite[Proposition 1]{NW99}):
\begin{align}
Li^\ell_1(z;\gamma,\sigma)=\rho_{1-z,\gamma}(\sigma).
\end{align}
Therefore, 
it holds that
\begin{align}\label{111l}
Li^\ell_{1,1,\dots,1}(z;\gamma,\sigma) = \frac{1}{n!} \left( \rho_{1-z,\gamma}(\sigma) \right)^n.
\end{align}
\end{rem}

The $\ell$-adic Galois multiple polylogarithm is analogous to the complex multiple polylogarithm, as summarized in Table \ref{table}.

\renewcommand{\arraystretch}{2.5}
\tabcolsep = 0.3cm
\begin{table}[htb]
\centering
  \caption{Analogy between the $\ell$-adic Galois and complex settings}
  \label{table}
  \begin{tabular}{|c||c|}  \hline 
    $\ell$-adic Galois side & Complex side  \\ \hline \hline

$z$ : $K$-rational base point on $\mathbb{P}^1 \backslash \{0,1,\infty\}$ & $z$ : $\mathbb{C}$-rational base point on $\mathbb{P}^1 \backslash \{0,1,\infty\}$ \\ \hline

$(\gamma,\sigma) \in \pi_1^\ellet \left(\mathbb{P}^1_{\overline{K}} \backslash \{0,1,\infty\};\overrightarrow{01},{z}\right) \times G_K$ & $\gamma \in \pi_1^{\rm top}\left(\mathbb{P}^1(\mathbb{C}) \backslash \{0,1,\infty\}; \overrightarrow{01}, z\right)$ \\ \hline

$e_0^{\rm B}={\bf log}(l_0),~e_1^{\rm B}={\bf log}(l_1)$ & $e_0^{\rm dR}=\left(\dfrac{dz}{z}\right)^{\ast},~e_1^{\rm dR}=\left(\dfrac{dz}{z-1}\right)^{\ast}$ \\ 
\hline

$e_\infty^{\rm B}={\bf log}\left({\bf exp}\left(-e_1^{\rm B}\right){\bf exp}\left(-e_0^{\rm B}\right)\right),$ 
& $e_\infty^{\rm dR}=-e_1^{\rm dR}-e_0^{\rm dR},$ \\ 

$l_\infty=l_1^{-1} l_0^{-1} \in \pi_1\left(\mathbb{P}^1 \backslash \{0,1,\infty\},\overrightarrow{01}\right)$ 
& $\overline{l}_\infty=-\overline{l}_1-\overline{l}_0 \in \pi_1^{\rm ab}\left(\mathbb{P}^1 \backslash \{0,1,\infty\},\overrightarrow{01}\right)$  \\ 
\hline

    ${\mathfrak f}^{z,\gamma}_{\sigma}\left(e_0^{\mathrm{B}},
e_1^{\mathrm{B}}\right) \in \mathbb{Q}_{\ell} \langle \langle e_0^{\mathrm{B}},
e_1^{\mathrm{B}} \rangle \rangle~~(\sigma \in G_K)$ & $G_{\overrightarrow{01}}^{z, \gamma}\left(e_0^{\rm dR},e_1^{\rm dR}\right) \in \mathbb{C} \langle \langle e_0^{\rm dR},e_1^{\rm dR} \rangle \rangle$ \\ \hline

  ${Li}^{\ell}_{\mathbf{k}}(z;\gamma,\sigma) \in  \mathbb{Q}_{\ell}$ & $Li_{\mathbf{k}}(z;\gamma) \in  \mathbb{C}$ \\ \hline

  $Li^\ell_{{\scriptsize \underbrace{1,1,\dots,1}_{n~\text{times}}}}(z;\gamma,\sigma) = \dfrac{1}{n!} \left( \rho_{1-z,\gamma}(\sigma) \right)^n$ & $Li_{{\scriptsize \underbrace{1,1,\dots,1}_{n~\text{times}}}}(z;\gamma) = \dfrac{1}{n!} \left( -\log(1-z;\gamma) \right)^n$  \\ \hline

  \end{tabular}
\end{table}

\section{Proofs of main results}\label{pfsec}
In this section, we provide the proofs for the main theorems.

\begin{proof}[Proof of Theorem \ref{M1}]
Let $\mathbf{k}\in\bigcup_{d=1}^{\infty}\mathbb{N}^d$.
Recall from (\ref{Lambda}) and (\ref{KZsol}) that $G_{\overrightarrow{01}}^{z,\gamma}=\bigl(\Lambda_{\overrightarrow{01}}^{z,\gamma}\bigr)^{\op}$, where the opposite anti-automorphism $(\cdot)^{\op}$ reverses each word, so that $(AB)^{\op}=B^{\op}A^{\op}$.
By Chen's multiplicativity of iterated integrals along the composite path $\gamma'=\delta\cdot\phi(\gamma)$ in (\ref{gamma'}),
\[
\Lambda_{\overrightarrow{01}}^{\frac{z}{z-1},\gamma'}\left(e_0^{\rm dR},e_1^{\rm dR}\right)
=\Lambda_{\overrightarrow{01}}^{\overrightarrow{0\infty},\delta}\left(e_0^{\rm dR},e_1^{\rm dR}\right) \cdot \Lambda_{\overrightarrow{0\infty}}^{\frac{z}{z-1},\phi(\gamma)}\left(e_0^{\rm dR},e_1^{\rm dR}\right);
\]
applying $(\cdot)^{\op}$, which reverses the two factors, gives
\begin{equation}\label{KZ_comp}
G_{\overrightarrow{01}}^{\frac{z}{z-1},\gamma'}\!\left(e_0^{\rm dR},e_1^{\rm dR}\right)
=G_{\overrightarrow{0\infty}}^{\frac{z}{z-1},\phi(\gamma)}\!\left(e_0^{\rm dR},e_1^{\rm dR}\right)\cdot G_{\overrightarrow{01}}^{\overrightarrow{0\infty},\delta}\!\left(e_0^{\rm dR},e_1^{\rm dR}\right).
\end{equation}
We evaluate the two factors on the right-hand side of (\ref{KZ_comp}).

\medskip
\noindent\emph{The first factor.}
Since $\phi$ maps $\overrightarrow{01}$ to $\overrightarrow{0\infty}$, naturality of the iterated integrals under pullback along $\phi$ gives
\[
\Lambda_{\overrightarrow{0\infty}}^{\frac{z}{z-1},\phi(\gamma)}\!\left(e_0^{\rm dR},e_1^{\rm dR}\right)
=1+\sum_{i=1}^{\infty}\int_{\overrightarrow{01},\gamma}^{z}\underbrace{\phi^\ast(\omega)\circ\cdots\circ\phi^\ast(\omega)}_{i~\text{times}}.
\]
By $\phi^\ast\!\left(\frac{dz}{z}\right)=\frac{dz}{z}-\frac{dz}{z-1}$, $\phi^\ast\!\left(\frac{dz}{z-1}\right)=-\frac{dz}{z-1}$ and $e_\infty^{\rm dR}=-e_1^{\rm dR}-e_0^{\rm dR}$ from (\ref{Cvariablerel}), we compute
\[
\phi^\ast(\omega)=\frac{dz}{z}e_0^{\rm dR}+\frac{dz}{z-1}e_\infty^{\rm dR},
\]
that is, $\phi^\ast(\omega)$ is obtained from $\omega$ by the substitution $e_1^{\rm dR}\mapsto e_\infty^{\rm dR}$.
As (\ref{Lambda}) depends on $e_0^{\rm dR},e_1^{\rm dR}$ only through the coefficients of $\omega$, this yields
\[
\Lambda_{\overrightarrow{0\infty}}^{\frac{z}{z-1},\phi(\gamma)}\!\left(e_0^{\rm dR},e_1^{\rm dR}\right)
=\Lambda_{\overrightarrow{01}}^{z,\gamma}\!\left(e_0^{\rm dR},e_\infty^{\rm dR}\right).
\]
Since $e_\infty^{\rm dR}$ is of degree one, the substitution $e_1^{\rm dR}\mapsto e_\infty^{\rm dR}$ commutes with $(\cdot)^{\op}$; applying $(\cdot)^{\op}$ therefore gives
\[
G_{\overrightarrow{0\infty}}^{\frac{z}{z-1},\phi(\gamma)}\!\left(e_0^{\rm dR},e_1^{\rm dR}\right)
=G_{\overrightarrow{01}}^{z,\gamma}\!\left(e_0^{\rm dR},e_\infty^{\rm dR}\right).
\]

\medskip
\noindent\emph{The second factor.}
The path $\delta$ is a counterclockwise half-turn around $0$, so $\int_\delta\frac{dz}{z}=\pi\sqrt{-1}$ and
\[
G_{\overrightarrow{01}}^{\overrightarrow{0\infty},\delta}\!\left(e_0^{\rm dR},e_1^{\rm dR}\right)={\bf exp}\!\left(\pi\sqrt{-1}\,e_0^{\rm dR}\right).
\]

\medskip
Substituting these into (\ref{KZ_comp}) yields the chain rule of complex KZ solutions (cf.\ \cite[Lemma 3.3 (2)]{NS25})
\begin{equation}\label{chainr1}
G_{\overrightarrow{01}}^{\frac{z}{z-1},\gamma'}\!\left(e_0^{\rm dR},e_1^{\rm dR}\right)
=G_{\overrightarrow{01}}^{z,\gamma}\!\left(e_0^{\rm dR},e_\infty^{\rm dR}\right)\cdot{\bf exp}\!\left(\pi\sqrt{-1}\,e_0^{\rm dR}\right).
\end{equation}
As ${\bf exp}\!\left(\pi\sqrt{-1}\,e_0^{\rm dR}\right)$ contributes only monomials ending in $e_0^{\rm dR}$, it leaves the coefficient of $W_\mathbf{k}= {\underbrace{e_0^{\rm dR} \ldots e_0^{\rm dR}}_{k_d-1~\text{times}}}e_1^{\rm dR} \cdots {\underbrace{e_0^{\rm dR} \ldots e_0^{\rm dR}}_{k_1-1~\text{times}}}e_1^{\rm dR}$ (which ends in $e_1^{\rm dR}$) unchanged; comparing this coefficient on both sides of (\ref{chainr1}) gives
\[
{\rm Coeff}_{W_{\mathbf{k}}}\!\left(G_{\overrightarrow{01}}^{\frac{z}{z-1},\gamma'}\!\left(e_0^{\rm dR},e_1^{\rm dR}\right)\right)
={\rm Coeff}_{W_{\mathbf{k}}}\!\left(G_{\overrightarrow{01}}^{z,\gamma}\!\left(e_0^{\rm dR},e_\infty^{\rm dR}\right)\right).
\]
Combining this with (\ref{formcpoly}) and Theorem \ref{thm22} gives the desired formula (\ref{Main1}).
\end{proof}

\begin{proof}[Proof of Theorem \ref{M2}]
Let $\mathbf{k}\in\bigcup_{d=1}^{\infty}\mathbb{N}^d$ and $\sigma\in G_K$.
The proof runs parallel to that of Theorem \ref{M1}, with the complex KZ solution replaced by the $\ell$-adic Galois $1$-cocycle (\ref{f2}).
Since the automorphism $\phi(z) = \frac{z}{z-1}$ is defined over $\mathbb{Q}$ (and thus over $K$), the Galois action commutes with $\phi$, meaning $\sigma(\phi(\gamma)) = \phi(\sigma(\gamma))$; hence, by (\ref{f}),
\begin{align}\label{chain_l2}
{\mathfrak f}_{\sigma}^{\frac{z}{z-1},\gamma'}
&=\gamma'\cdot\sigma(\gamma')^{-1}
=\delta\cdot\phi(\gamma)\cdot\phi\!\left(\sigma(\gamma)\right)^{-1}\cdot\sigma(\delta)^{-1}\\
&=\delta\cdot\phi\!\left({\mathfrak f}_{\sigma}^{z,\gamma}\right)\cdot\delta^{-1}\cdot\delta\cdot\sigma(\delta)^{-1}
=\left(\delta\cdot\phi\!\left({\mathfrak f}_{\sigma}^{z,\gamma}\right)\cdot\delta^{-1}\right)\cdot{\mathfrak f}_{\sigma}^{\overrightarrow{0\infty},\delta},\notag
\end{align}
the pro-$\ell$ analogue of (\ref{KZ_comp}), where ${\mathfrak f}_{\sigma}^{\overrightarrow{0\infty},\delta}=\delta\cdot\sigma(\delta)^{-1}$. 
We now evaluate the two factors on the right-hand side of (\ref{chain_l2}).

\medskip
\noindent\emph{The first factor.}
The map $\Phi\colon{\mathfrak f}\mapsto\delta\cdot\phi({\mathfrak f})\cdot\delta^{-1}$ is an automorphism of $\pi_1^{\ellet}\!\left({\bP}^{1}_{\overline{K}}\backslash\{0,1,\infty\},\overrightarrow{01}\right)$.
Then,
\[
\Phi(l_0) = \delta \cdot \phi(l_0) \cdot \delta^{-1} = l_0,
\quad
\Phi(l_1) = \delta \cdot \phi(l_1) \cdot \delta^{-1} = l_\infty.
\]
Thus $\Phi$ substitutes $l_1\mapsto l_\infty$ in ${\mathfrak f}_{\sigma}^{z,\gamma}={\mathfrak f}_{\sigma}^{z,\gamma}(l_0,l_1)$
where
${\mathfrak f}(l_\ast,l_{\ast'})$ is the image of ${\mathfrak f}$ under the map
\[
\pi_1^{{\ell\text{-\'et}}}\left({\bP}^{1}_{\overline{K}} \backslash \{0,1,\infty\}, \overrightarrow{01}\right) \to \pi_1^{{\ell\text{-\'et}}}\left({\bP}^{1}_{\overline{K}} \backslash \{0,1,\infty\}, \overrightarrow{01}\right)
\]
given by $l_0,l_1 \mapsto l_\ast,l_{\ast'}$, respectively.
Then, we obtain
\[
\delta\cdot\phi\!\left({\mathfrak f}_{\sigma}^{z,\gamma}\right)\cdot\delta^{-1}
={\mathfrak f}_{\sigma}^{z,\gamma}(l_0,l_\infty).
\]
Thus,
recalling $e_\infty^{\rm B} = {\bf log}(l_\infty)$ from (\ref{Bvari}) and the notation of (\ref{f2}),
the first factor on the right-hand side of (\ref{chain_l2}) becomes
${\mathfrak f}_{\sigma}^{z,\gamma}\!\left(e_0^{\rm B}, e_\infty^{\rm B}\right)$.

\medskip
\noindent\emph{The second factor.}
The Galois action on $\delta$ is $\sigma(\delta)=l_0^{\frac{\chi(\sigma)-1}{2}}\delta$ (cf.\ \cite[3.2]{I90}) where $\chi\colon G_K\to\mathbb{Z}_{\ell}^\times$ is the $\ell$-adic cyclotomic character, hence
\[
{\mathfrak f}_{\sigma}^{\overrightarrow{0\infty},\delta}\!\left(e_0^{\rm B},e_1^{\rm B}\right)
=\delta\cdot\sigma(\delta)^{-1}
=l_0^{\frac{1-\chi(\sigma)}{2}}
={\bf exp}\!\left(\tfrac{1-\chi(\sigma)}{2}\,e_0^{\rm B}\right).
\]

\medskip
Substituting these into (\ref{chain_l2}) yields the chain rule of $\ell$-adic Galois $1$-cocycles (cf.\ \cite[Lemma 3.3 (4)]{NS25})
\begin{equation}\label{chainr2}
{\mathfrak f}_{\sigma}^{\frac{z}{z-1},\gamma'}\!\left(e_0^{\rm B},e_1^{\rm B}\right)
={\mathfrak f}_{\sigma}^{z,\gamma}\!\left(e_0^{\rm B},e_\infty^{\rm B}\right)\cdot{\bf exp}\!\left(\tfrac{1-\chi(\sigma)}{2}\,e_0^{\rm B}\right),
\end{equation}
the $\ell$-adic Galois analogue of (\ref{chainr1}).
The exponential factor ${\bf exp}\!\left(\tfrac{1-\chi(\sigma)}{2}\,e_0^{\rm B}\right)$ contributes only monomials ending in $e_0^{\rm B}$, so comparing the coefficients of $W_\mathbf{k}= {\underbrace{e_0^{\rm B} \ldots e_0^{\rm B}}_{k_d-1~\text{times}}}e_1^{\rm B} \cdots {\underbrace{e_0^{\rm B} \ldots e_0^{\rm B}}_{k_1-1~\text{times}}}e_1^{\rm B}$ (which ends in $e_1^{\rm B}$) on both sides of (\ref{chainr2}) gives
\[
{\rm Coeff}_{W_{\mathbf{k}}}\!\left({\mathfrak f}_{\sigma}^{\frac{z}{z-1},\gamma'}\!\left(e_0^{\rm B},e_1^{\rm B}\right)\right)
={\rm Coeff}_{W_{\mathbf{k}}}\!\left({\mathfrak f}_{\sigma}^{z,\gamma}\!\left(e_0^{\rm B},e_\infty^{\rm B}\right)\right).
\]
Combining this identity with (\ref{lpoly1}) and Theorem \ref{thm33},
we obtain the desired formula (\ref{Main2}).
\end{proof}

\renewcommand{\arraystretch}{2.2}
\tabcolsep = 0.2cm
\begin{table}[htb]
\centering
  \caption{Chain rules}
  \label{table2}
  \begin{tabular}{|c||c|}  \hline 
    $\ell$-adic Galois side & ${\mathfrak f}_{\sigma}^{\frac{z}{z-1},\gamma'}\left(e_0^{\rm B},
e_1^{\rm B}\right)
={\mathfrak f}_{\sigma}^{z,\gamma}\left(e_0^{\rm B},
e_\infty^{\rm B}\right) \cdot {\bf exp}\left(\dfrac{1-\chi(\sigma)}{2} e_0^{\rm B}\right)$,  \\ 

   & $e_\infty^{\rm B}={\bf log}\left({\bf exp}\left(-e_1^{\rm B}\right){\bf exp}\left(-e_0^{\rm B}\right)\right)$  \\ \hline \hline

Complex side & $G_{\overrightarrow{01}}^{\frac{z}{z-1}, \gamma'}\left(e_0^{\rm dR},e_1^{\rm dR}\right)=G_{\overrightarrow{01}}^{z, \gamma}\left(e_0^{\rm dR},e_\infty^{\rm dR}\right)\cdot {\bf exp}\left(\pi \sqrt{-1} e_0^{\rm dR}\right)$, \\ 

 & $e_\infty^{\rm dR}=-e_1^{\rm dR}-e_0^{\rm dR}$  \\ 
\hline
  \end{tabular}
\end{table}

\begin{proof}[Proof of Theorem \ref{M3}]
Let $n \in \mathbb{N}$ and $\sigma \in G_K$.
By the chain rule (\ref{chainr2}) and Definition \ref{lpolydef}, we obtain
\begin{equation}\label{pf1.4-1}
\operatorname{Coeff}_{n}\left({\mathfrak f}_{\sigma}^{z,\gamma}\left(e_0^{\rm B}, e_\infty^{\rm B}\right)\right) 
= -Li^{\ell}_n\left(\frac{z}{z-1};\gamma',\sigma\right).
\end{equation}
On the other hand, Corollary \ref{cor:single_index} implies
\begin{equation}\label{pf1.4-2}
\operatorname{Coeff}_{n}\left({\mathfrak f}_{\sigma}^{z,\gamma}\left(e_0^{\rm B}, e_\infty^{\rm B}\right)\right) 
= \sum_{j=0}^{n-1} \frac{(-1)^{j} B_{j}}{j!} \left(\sum_{\substack{\mathbf{J} \in \mathbb{N}^d,~d \in \mathbb{N} \\ {\rm wt}(\mathbf{J})=n-j}} {Li}^{\ell}_{\mathbf{J}}(z;\gamma,\sigma)\right).
\end{equation}
Combining (\ref{pf1.4-1}) and (\ref{pf1.4-2}), we obtain the relation
\begin{equation}\label{pf1.4-3}
-Li^{\ell}_n\left(\frac{z}{z-1};\gamma',\sigma\right)
= \sum_{j=0}^{n-1} \frac{(-1)^{j} B_{j}}{j!} S(n-j),
\end{equation}
where $\displaystyle S(m) := \sum_{\substack{\mathbf{J} \in \mathbb{N}^d,~d \in \mathbb{N} \\ {\rm wt}(\mathbf{J})=m}} {Li}^{\ell}_{\mathbf{J}}(z;\gamma,\sigma)$.
Let us introduce the generating functions
\[
\mathcal{L}(t) := \sum_{n=1}^\infty \left(-Li^{\ell}_n\left(\frac{z}{z-1};\gamma',\sigma\right)\right) t^n, \quad
\mathcal{S}(t) := \sum_{n=1}^\infty S(n) t^n.
\]
The identity (\ref{pf1.4-3}) is equivalent to the product $\mathcal{L}(t) = \mathcal{B}(t) \mathcal{S}(t)$, where $\mathcal{B}(t)$ is the generating function of the Bernoulli numbers given by
\[
\mathcal{B}(t) := \sum_{j=0}^\infty \frac{B_j}{j!} (-t)^j = \frac{t e^t}{e^t-1}.
\]
Since $\mathcal{B}(t)^{-1} = \frac{e^{-t}-1}{-t} = \sum_{k=0}^\infty \frac{(-1)^k}{(k+1)!} t^k$, we have
\[
\mathcal{S}(t) = \mathcal{B}(t)^{-1} \mathcal{L}(t) = \left( \sum_{k=0}^\infty \frac{(-1)^k}{(k+1)!} t^k \right) \mathcal{L}(t).
\]
Comparing the coefficient of $t^n$ on both sides immediately leads to the desired formula.
\end{proof}

\section{Explicit $\ell$-adic Landen formula for low weight indices} \label{sec:low_weight_examples}

In this section, we derive explicit forms of the $\ell$-adic Landen formula for some specific indices $\mathbf{k}$ of low weight.
We illustrate Theorem \ref{M2} by applying it to the cases
$\mathbf{k}=(2)$, $\mathbf{k}=(3)$, $\mathbf{k}=(4)$, $\mathbf{k}=(1,2)$, $\mathbf{k}=(2,1)$, $\mathbf{k}=(1,1,2)$, and $\mathbf{k}=(2,3)$. Recall that the formula is given by:
\begin{equation} \label{Main2_recap}
    Li^{\ell}_{\mathbf{k}}\left(\frac{z}{z-1};\gamma',\sigma\right) + (-1)^{1+\depth(\mathbf{k})} \sum_{\mathbf{J} \succeq \mathbf{k}} Li^{\ell}_{\mathbf{J}}\left(z;\gamma,\sigma\right) = \mathcal{E}_{\mathbf{k}}(z;\gamma,\sigma),
\end{equation}
where $\mathcal{E}_{\mathbf{k}}(z;\gamma,\sigma)$ denotes the error term,
i.e.,
the right-hand side of Theorem \ref{M2}.

\subsection{Case $\mathbf{k}=(2)$ (Weight 2, Depth 1)}
Let $\mathbf{k}=(2)$. The associated non-commutative monomial is $W_{\mathbf{k}} = e_0 e_1$.

\paragraph{\underline{Left-hand side}}
The sign factor is $(-1)^{1+\depth(\mathbf{k})} = +1$.
The refinements $\mathbf{J} \succeq (2)$ correspond to partitions of the integer $2$, which are $\mathbf{J}=(2)$ and $\mathbf{J}=(1,1)$.
Thus, the left-hand side is:
\[
    Li^{\ell}_{2}\left(\frac{z}{z-1};\gamma',\sigma\right) + \left( Li^{\ell}_{2}(z;\gamma,\sigma) + Li^{\ell}_{1,1}(z;\gamma,\sigma) \right).
\]

\paragraph{\underline{Right-hand side}}
Since $\depth(\mathbf{k})=1$, only the partition length $M=1$ contributes to the error term.
\begin{itemize}
    \item \textbf{Case $M=1$:} The global sign corresponds to $(-1)^{M+\depth(\mathbf{k})} = +1$.
    We solve the equation $e_0^{S_1-1} W_{\mathbf{u}_1} e_0^{r_1} = e_0 e_1$. 
    Since $r_1$ must be $0$, we have $e_0^{S_1-1} W_{\mathbf{u}_1} = e_0 e_1$.
    \begin{itemize}
        \item If $S_1=2$, then $W_{\mathbf{u}_1}=e_1$, implying $(\mathbf{u}_1, r_1)=((1),0)$. This is the trivial pair.
Therefore, this case is excluded from the sum.
        \item If $S_1=1$, then $W_{\mathbf{u}_1}=e_0 e_1$, implying $\mathbf{u}_1=(2)$. Thus $\mathbf{S}=(1)$.
        The coefficient is $\mu_{(2), 0} = -1/2$.
        The refinement of $\mathbf{S}=(1)$ is simply $(1)$.
        The contribution is:
        \[ (+1) \cdot \left( Li^{\ell}_1(z;\gamma,\sigma) \right) \cdot \left( -\frac{1}{2} \right) = -\frac{1}{2} Li^{\ell}_1(z;\gamma,\sigma). \]
    \end{itemize}
\end{itemize}

Combining these results, we obtain the following formula. Note that the error term $-\frac{1}{2} Li^{\ell}_1(z;\gamma,\sigma)$ can be interpreted as $\frac{1}{2} Li^{\ell}_1(\frac{z}{z-1};\gamma',\sigma)$ using the relation $Li^{\ell}_1(\frac{z}{z-1};\gamma',\sigma) = -Li^{\ell}_1(z;\gamma,\sigma)$, which aligns with the case $n=2$ of Theorem \ref{M3}.
\begin{prop}[Explicit $\ell$-adic Landen formula for $\mathbf{k}=(2)$]\label{sceq2}
Let the notation and assumptions be as in Theorem \ref{M2}.
For any $\sigma \in G_K$, the following holds:
\begin{align} \label{eq:k2}
    Li^{\ell}_{2}\left(\dfrac{z}{z-1};\gamma',\sigma\right) + Li^{\ell}_{2}(z;\gamma,\sigma) + Li^{\ell}_{1,1}(z;\gamma,\sigma) &= -\frac{1}{2} Li^{\ell}_{1}(z;\gamma,\sigma)\\
&=\frac{1}{2} Li^{\ell}_1\left(\frac{z}{z-1};\gamma',\sigma\right). \notag
\end{align}
\end{prop}

\begin{rem}
By using the relation $Li^\ell_{1,1,\dots,1}(z;\gamma,\sigma) = \frac{1}{n!} \left( \rho_{1-z,\gamma}(\sigma) \right)^n$ from Remark~\ref{rem:shuffle_Li1}, the explicit formula \eqref{eq:k2} in Proposition~\ref{sceq2} can be rewritten solely in terms of $\rho_{1-z,\gamma}(\sigma)$ for its depth $1$ components as follows:
\begin{align}
Li^{\ell}_{2}\left(\dfrac{z}{z-1};\gamma',\sigma\right) + Li^{\ell}_{2}(z;\gamma,\sigma) + \frac{1}{2} \left( \rho_{1-z,\gamma}(\sigma) \right)^2 = -\frac{1}{2} \rho_{1-z,\gamma}(\sigma).
\end{align}
This formula
is nothing but the functional equation of the $\ell$-adic Galois dilogarithm \cite[(8)]{NS25}, \cite[(6.22)]{NW12}.
\end{rem}

\begin{rem}
\begin{proof}[Alternative proof of Proposition \ref{sceq2} by direct comparison:]
Recall the chain rule of $\ell$-adic Galois 1-cocycles (\ref{chainr2}):
\[
{\mathfrak f}_{\sigma}^{\frac{z}{z-1},\gamma'}\left(e_0^{\rm B}, e_1^{\rm B}\right) = {\mathfrak f}_{\sigma}^{z,\gamma}\left(e_0^{\rm B}, e_\infty^{\rm B}\right) \cdot {\bf exp}\left(c e_0^{\rm B}\right) \quad \left(c := \frac{1-\chi(\sigma)}{2}\right).
\]
We expand the formal power series ${\mathfrak f}_{\sigma}^{z,\gamma}(e_0^{\rm B}, e_1^{\rm B})$ up to degree 2:
\[
{\mathfrak f}_{\sigma}^{z,\gamma}(e_0^{\rm B}, e_1^{\rm B}) = 1 + f_{e_0}e_0^{\rm B} + f_{e_1}e_1^{\rm B} + f_{e_0^2}(e_0^{\rm B})^2 + f_{e_0e_1}e_0^{\rm B}e_1^{\rm B} + f_{e_1e_0}e_1^{\rm B}e_0^{\rm B} + f_{e_1^2}(e_1^{\rm B})^2 + (\text{deg} \ge 3),
\]
where $f_{w} := {\rm Coeff}_{w}({\mathfrak f}_{\sigma}^{z,\gamma}(e_0^{\rm B}, e_1^{\rm B}))$. By Definition \ref{lpolydef}, the coefficients corresponding to monomials ending in $e_1^{\rm B}$ are:
\[
f_{e_1} = -Li^{\ell}_1(z;\gamma,\sigma), \quad f_{e_0e_1} = -Li^{\ell}_2(z;\gamma,\sigma), \quad f_{e_1^2} = Li^{\ell}_{1,1}(z;\gamma,\sigma).
\]
The variable $e_\infty^{\rm B}$ expands via the Baker-Campbell-Hausdorff formula as:
\[
e_\infty^{\rm B} = {\bf log}\left({\bf exp}\left(-e_1^{\rm B}\right){\bf exp}\left(-e_0^{\rm B}\right)\right) = -e_1^{\rm B} - e_0^{\rm B} + \frac{1}{2}e_1^{\rm B}e_0^{\rm B} - \frac{1}{2}e_0^{\rm B}e_1^{\rm B} + (\text{deg} \ge 3).
\]
Substituting this expression into ${\mathfrak f}_{\sigma}^{z,\gamma}(e_0^{\rm B}, e_\infty^{\rm B})$, we extract the terms contributing to the coefficient of the monomial $e_0^{\rm B} e_1^{\rm B}$:
\begin{align*}
f_{e_1} e_\infty^{\rm B} &\equiv f_{e_1} \left( -e_1^{\rm B} - e_0^{\rm B} - \frac{1}{2}e_0^{\rm B}e_1^{\rm B} + \frac{1}{2}e_1^{\rm B}e_0^{\rm B} \right) \pmod{\text{deg} \ge 3}, \\
f_{e_0e_1} e_0^{\rm B} e_\infty^{\rm B} &\equiv f_{e_0e_1} e_0^{\rm B} (-e_1^{\rm B} - e_0^{\rm B}) = -f_{e_0e_1} e_0^{\rm B}e_1^{\rm B} - f_{e_0e_1} (e_0^{\rm B})^2 \pmod{\text{deg} \ge 3}, \\
f_{e_1^2} (e_\infty^{\rm B})^2 &\equiv f_{e_1^2} (-e_1^{\rm B} - e_0^{\rm B})^2 = f_{e_1^2} \left((e_1^{\rm B})^2 + e_1^{\rm B}e_0^{\rm B} + e_0^{\rm B}e_1^{\rm B} + (e_0^{\rm B})^2\right) \pmod{\text{deg} \ge 3}.
\end{align*}
Summing these specific contributions, the coefficient of $e_0^{\rm B} e_1^{\rm B}$ in ${\mathfrak f}_{\sigma}^{z,\gamma}(e_0^{\rm B}, e_\infty^{\rm B})$ is exactly:
\[
{\rm Coeff}_{e_0^{\rm B} e_1^{\rm B}}\left({\mathfrak f}_{\sigma}^{z,\gamma}(e_0^{\rm B}, e_\infty^{\rm B})\right) = -\frac{1}{2} f_{e_1} - f_{e_0e_1} + f_{e_1^2}.
\]
Notice that multiplying by ${\bf exp}(c e_0^{\rm B}) = 1 + c e_0^{\rm B} + \cdots$ on the right-hand side only produces additional terms ending in $e_0^{\rm B}$. Therefore, this exponential factor does not affect the coefficient of $e_0^{\rm B} e_1^{\rm B}$. Comparing the coefficient of $e_0^{\rm B} e_1^{\rm B}$ on both sides of the chain rule (\ref{chainr2}) gives:
\[
{\rm Coeff}_{e_0^{\rm B} e_1^{\rm B}}\left({\mathfrak f}_{\sigma}^{\frac{z}{z-1},\gamma'}(e_0^{\rm B}, e_1^{\rm B})\right) = -\frac{1}{2} f_{e_1} - f_{e_0e_1} + f_{e_1^2}.
\]
Substituting the corresponding polylogarithm values $-Li^{\ell}_2(\frac{z}{z-1};\gamma',\sigma)$ on the left side and the values of $f_{w}$ on the right side, we obtain:
\[
-Li^{\ell}_2\left(\frac{z}{z-1};\gamma',\sigma\right) = -\frac{1}{2}\left(-Li^{\ell}_1(z;\gamma,\sigma)\right) - \left(-Li^{\ell}_2(z;\gamma,\sigma)\right) + Li^{\ell}_{1,1}(z;\gamma,\sigma).
\]
Rearranging this equality explicitly recovers the formula (\ref{eq:k2}):
\[
Li^{\ell}_{2}\left(\dfrac{z}{z-1};\gamma',\sigma\right) + Li^{\ell}_{2}(z;\gamma,\sigma) + Li^{\ell}_{1,1}(z;\gamma,\sigma) = -\frac{1}{2} Li^{\ell}_{1}(z;\gamma,\sigma).
\]
\end{proof}
\end{rem}

\subsection{Case $\mathbf{k}=(3)$ (Weight 3, Depth 1)}
Let $\mathbf{k}=(3)$. The associated monomial is $W_{\mathbf{k}} = e_0^2 e_1$.

\paragraph{\underline{Left-hand side}}
The sign factor is $(-1)^{1+\depth(\mathbf{k})}=+1$.
The refinements of $\mathbf{k}=(3)$ are $(3)$, $(2,1)$, $(1,2)$, and $(1,1,1)$.
Thus, the left-hand side is:
\[
    Li^{\ell}_{3}\left(\frac{z}{z-1};\gamma',\sigma\right) + \left( Li^{\ell}_{3}(z;\gamma,\sigma) + Li^{\ell}_{2,1}(z;\gamma,\sigma) + Li^{\ell}_{1,2}(z;\gamma,\sigma) + Li^{\ell}_{1,1,1}(z;\gamma,\sigma) \right).
\]

\paragraph{\underline{Right-hand side}}
Similarly to the previous case, only $M=1$ contributes.
\begin{itemize}
    \item \textbf{Case $M=1$:}
    The global sign corresponds to $(-1)^{M+\depth(\mathbf{k})} = +1$.
    We consider the equation $e_0^{S_1-1} W_{\mathbf{u}_1} = e_0^2 e_1$.
    \begin{itemize}
        \item If $S_1=2$, then~$\mathbf{S}=(2)$ and $W_{\mathbf{u}_1}=e_0 e_1$, i.e., $\mathbf{u}_1=(2)$.
        Coefficient: $\mu_{(2), 0} = -1/2$.
        Refinements of $\mathbf{S}=(2)$ are $(2)$ and $(1,1)$.
        Contribution:
        \[ (+1) \cdot \left( Li^{\ell}_2(z;\gamma,\sigma) + Li^{\ell}_{1,1}(z;\gamma,\sigma) \right) \cdot \left( -\frac{1}{2} \right). \]
        
        \item If $S_1=1$, then~$\mathbf{S}=(1)$ and $W_{\mathbf{u}_1}=e_0^2 e_1$, i.e., $\mathbf{u}_1=(3)$.
        By Proposition \ref{prop:integral_rep_corrected2}, 
        \[ \mu_{(3), 0} = \frac{(-1)^3 B_2}{2!} = \frac{-1 \cdot (1/6)}{2} = -\frac{1}{12}. \]
        Refinement of $\mathbf{S}=(1)$ is $(1)$.
        Contribution:
        \[ (+1) \cdot \left( Li^{\ell}_1(z;\gamma,\sigma) \right) \cdot \left( -\frac{1}{12} \right). \]
    \end{itemize}
\end{itemize}

\noindent
By the above discussion and Theorem \ref{M3},
we obtain the following formula.

\begin{prop}[Explicit $\ell$-adic Landen formula for $\mathbf{k}=(3)$]\label{sceq3}
Let the notation and assumptions be as in Theorem \ref{M2}.
For any $\sigma \in G_K$, the following holds:
\begin{align}\label{explL3}
    & Li^{\ell}_{3}\left(\dfrac{z}{z-1};\gamma',\sigma\right) + Li^{\ell}_{3}(z;\gamma,\sigma) + Li^{\ell}_{2,1}(z;\gamma,\sigma) + Li^{\ell}_{1,2}(z;\gamma,\sigma) + Li^{\ell}_{1,1,1}(z;\gamma,\sigma) \\
    &\quad \quad \quad = -\frac{1}{12} Li^{\ell}_{1}(z;\gamma,\sigma) - \frac{1}{2} Li^{\ell}_{2}(z;\gamma,\sigma) - \frac{1}{2} Li^{\ell}_{1,1}(z;\gamma,\sigma) \notag \\
   &\quad \quad \quad = -\frac{1}{6} Li^{\ell}_1\left(\frac{z}{z-1};\gamma',\sigma\right) + \frac{1}{2} Li^{\ell}_2\left(\frac{z}{z-1};\gamma',\sigma\right) \notag
\end{align}
\end{prop}

\begin{rem}
By applying the relationship
(\ref{111l})
between $\ell$-adic Galois multiple polylogarithms of index $(1,\dots,1)$ and the $\ell$-adic Kummer 1-cocycle $\rho_{1-z,\gamma}(\sigma)$, the identity \eqref{explL3} in Proposition~\ref{sceq3} can be expressed in the following alternative form:
\begin{align}\label{slan3}
& Li^{\ell}_{3}\left(\dfrac{z}{z-1};\gamma',\sigma\right) + Li^{\ell}_{3}(z;\gamma,\sigma) + Li^{\ell}_{2,1}(z;\gamma,\sigma) + Li^{\ell}_{1,2}(z;\gamma,\sigma) + \frac{1}{6} \left( \rho_{1-z,\gamma}(\sigma) \right)^3 \\
&\quad \quad \quad = -\frac{1}{12} \rho_{1-z,\gamma}(\sigma) - \frac{1}{2} Li^{\ell}_{2}(z;\gamma,\sigma) - \frac{1}{4} \left( \rho_{1-z,\gamma}(\sigma) \right)^2. \notag
\end{align}
In \cite[Proof of Theorem 1.1]{NS25},
the Landen type functional equation for the $\ell$-adic Galois trilogarithm
was derived from (\ref{slan3}) with the chain rule of $\ell$-adic Galois 1-cocycles arising from the symmetry $z \mapsto 1-z$ of $\mathbb{P}^1 \backslash \{0,1,\infty\}$.
\end{rem}

\subsection{Case $\mathbf{k}=(4)$ (Weight 4, Depth 1)}
Let $\mathbf{k}=(4)$. 
The associated monomial is $W_{\mathbf{k}} = e_0^3 e_1$.
\paragraph{\underline{Left-hand side}}
The sign factor is $(-1)^{1+\depth(\mathbf{k})}=+1$.
The refinements $\mathbf{J} \succeq (4)$ correspond to the partitions of the integer $4$. 
There are $2^{4-1}=8$ such refinements:
$$(4),~(3,1),~(2,2),~(1,3),~(2,1,1),~(1,2,1),~(1,1,2),~(1,1,1,1).$$
Thus, the sum on the left-hand side is:
\begin{align*}
&Li^{\ell}_{4}\left(\frac{z}{z-1};\gamma',\sigma\right) + \biggl( Li^{\ell}_{4}(z;\gamma,\sigma) + Li^{\ell}_{3,1}(z;\gamma,\sigma) + Li^{\ell}_{2,2}(z;\gamma,\sigma) + Li^{\ell}_{1,3}(z;\gamma,\sigma) \\
&\quad + Li^{\ell}_{2,1,1}(z;\gamma,\sigma) + Li^{\ell}_{1,2,1}(z;\gamma,\sigma) + Li^{\ell}_{1,1,2}(z;\gamma,\sigma) + Li^{\ell}_{1,1,1,1}(z;\gamma,\sigma) \biggr).\end{align*}
\paragraph{\underline{Right-hand side}}
Since $\depth(\mathbf{k})=1$, only the partition length $M=1$ contributes. 
We examine the equation $e_0^{S_1-1} W_{\mathbf{u}_1} = e_0^3 e_1$.
\begin{itemize}\item \textbf{Case $M=1$:} 
The global sign corresponds to $(-1)^{M+\depth(\mathbf{k})} = +1$.
We iterate through possible values for $S_1$ (determining $\mathbf{S}=(S_1)$).\begin{itemize}
\item If $S_1=4$, then $W_{\mathbf{u}_1}=e_1$, implying $\mathbf{u}_1=(1)$. 
This corresponds to the trivial pair which matches the main term structure and is excluded from the error summation.
\item If $S_1=3$, then $e_0^2 W_{\mathbf{u}_1} = e_0^3 e_1$, implying $W_{\mathbf{u}_1}=e_0 e_1$, so $\mathbf{u}_1=(2)$. Thus $\mathbf{S}=(3)$.
The coefficient is $\mu_{(2), 0} = -1/2$.
The refinements of $\mathbf{S}=(3)$ are $(3), (2,1), (1,2), (1,1,1)$.
The contribution is:
\[
(+1) \cdot \left( \sum_{\mathbf{J} \succeq (3)} Li^{\ell}_{\mathbf{J}}(z;\gamma,\sigma) \right) \cdot \left( -\frac{1}{2} \right).
\]

\item If $S_1=2$, then $e_0^1 W_{\mathbf{u}_1} = e_0^3 e_1$, implying $W_{\mathbf{u}_1}=e_0^2 e_1$, so $\mathbf{u}_1=(3)$. Thus $\mathbf{S}=(2)$.
The coefficient is $\mu_{(3), 0} = -1/12$.
The refinements of $\mathbf{S}=(2)$ are $(2), (1,1)$.
The contribution is:
\[
(+1) \cdot \left( \sum_{\mathbf{J} \succeq (2)} Li^{\ell}_{\mathbf{J}}(z;\gamma,\sigma) \right) \cdot \left( -\frac{1}{12} \right).
\]

\item If $S_1=1$, then $W_{\mathbf{u}_1} = e_0^3 e_1$, so $\mathbf{u}_1=(4)$. Thus $\mathbf{S}=(1)$.
The coefficient is $\mu_{(4), 0}$. Since this corresponds to the Bernoulli number term related to $B_3=0$, we have $\mu_{(4), 0} = 0$ by Proposition \ref{prop:integral_rep_corrected2}.
Thus, this term vanishes.
\end{itemize}\end{itemize}
Combining these results and Theorem \ref{M3}, we obtain the explicit formula for weight 4.
\begin{prop}[Explicit $\ell$-adic Landen formula for $\mathbf{k}=(4)$]\label{sceq4}
Let the notation and assumptions be as in Theorem \ref{M2}.
For any $\sigma \in G_K$, the following holds:
\begin{align}
&Li^{\ell}_{4}\left(\frac{z}{z-1};\gamma',\sigma\right) + \biggl( Li^{\ell}_{4}(z;\gamma,\sigma) + Li^{\ell}_{3,1}(z;\gamma,\sigma) + Li^{\ell}_{2,2}(z;\gamma,\sigma) + Li^{\ell}_{1,3}(z;\gamma,\sigma) \\
&\quad + Li^{\ell}_{2,1,1}(z;\gamma,\sigma) + Li^{\ell}_{1,2,1}(z;\gamma,\sigma) + Li^{\ell}_{1,1,2}(z;\gamma,\sigma) + Li^{\ell}_{1,1,1,1}(z;\gamma,\sigma) \biggr) \notag \\
& = \frac{1}{2} Li^{\ell}_3\left(\frac{z}{z-1};\gamma',\sigma\right) - \frac{1}{6} Li^{\ell}_2\left(\frac{z}{z-1};\gamma',\sigma\right) + \frac{1}{24} Li^{\ell}_1\left(\frac{z}{z-1};\gamma',\sigma\right) \notag \\
&= -\frac{1}{2} \left( Li^{\ell}_{3}(z;\gamma,\sigma) + Li^{\ell}_{2,1}(z;\gamma,\sigma) + Li^{\ell}_{1,2}(z;\gamma,\sigma) + Li^{\ell}_{1,1,1}(z;\gamma,\sigma) \right) \notag \\
&\quad - \frac{1}{12} \left( Li^{\ell}_{2}(z;\gamma,\sigma) + Li^{\ell}_{1,1}(z;\gamma,\sigma) \right) \notag 
\end{align}\end{prop}

\subsection{Case $\mathbf{k}=(1,2)$ (Weight 3, Depth 2)}
Let $\mathbf{k}=(1,2)$. The associated monomial is $W_{\mathbf{k}} = e_0 e_1 e_1$.

\paragraph{\underline{Left-hand side}}
The sign factor is $(-1)^{1+\depth(\mathbf{k})}=-1$.
Refinements of $(1,2)$ are $(1,2)$ and $(1,1,1)$.
Thus, the left-hand side is:
\[
    Li^{\ell}_{1,2}\left(\dfrac{z}{z-1};\gamma',\sigma\right) - \left( Li^{\ell}_{1,2}(z;\gamma,\sigma) + Li^{\ell}_{1,1,1}(z;\gamma,\sigma) \right).
\]

\paragraph{\underline{Right-hand side}}
\begin{itemize}
    \item \textbf{Case $M=1$:} 
    The global sign is $(-1)^{M+\depth(\mathbf{k})} = -1$.
    We solve the equation $e_0^{S_1-1} W_{\mathbf{u}_1} = e_0 e_1 e_1$.
    \begin{itemize}
        \item If $S_1=2$, then $W_{\mathbf{u}_1}=e_1 e_1$, i.e.,
$\mathbf{u}_1=(1,1)$. 
The coefficient $\mu_{(1,1),0}$ of the term $e_1^2$ in $e'_\infty$ vanishes. 
Thus, the contribution is 0.
        \item If $S_1=1$, then $W_{\mathbf{u}_1}=e_0 e_1 e_1$, so $\mathbf{u}_1=(1,2)$. Then $\mathbf{S}=(1)$.
        The coefficient is $\mu_{(1,2), 0}$.
        Using Proposition \ref{prop:integral_rep_corrected}, the index list for $\mathbf{u}=(1,2)$ is derived from $(2-1, 1, 1-1, 1, 0)=(1,1,0,1,0)$, which gives the non-zero indices $(1,2)$ (so $L=2$).
        \[ I = \int_0^1 t^{\lfloor 2/2 \rfloor} (t-1)^{\lfloor 1/2 \rfloor} G_1(t)G_2(t) \, dt = \int_0^1 t\left(t-\dfrac{1}{2}\right) \, dt = \frac{1}{12}. \]
        Since $u_m=u_2=2>1$, the sign $\sigma_{\mathbf{u}, r}$ in (\ref{GPint}) is $-1$. Thus $\mu_{(1,2), 0} = - 1/12$.
        Contribution:
        \[ (-1) \cdot \left( Li^{\ell}_1(z;\gamma,\sigma) \right) \cdot \left( -\frac{1}{12} \right) = \frac{1}{12} Li^{\ell}_1(z;\gamma,\sigma). \]
    \end{itemize}
    
    \item \textbf{Case $M=2$:} 
    The global sign is $(-1)^{M+\depth(\mathbf{k})} = +1$.
    We decompose $W_{(1,2)} = e_0 e_1 e_1$ into $(e_0^{S_2-1}W_{\mathbf{u}_2} e_0^{r_2}) \cdot (e_0^{S_1-1}W_{\mathbf{u}_1})$.
    The only possible split where each part ends in $e_1$ is:
    \[ \underbrace{e_0 e_1}_{\text{Left part } (i=2)} \cdot \underbrace{e_1}_{\text{Right part } (i=1)}. \]
    \begin{itemize}
        \item Left part ($i=2$): $e_0^{S_2-1} W_{\mathbf{u}_2} = e_0 e_1$. 
Hence $S_2=1, \mathbf{u}_2=(2)$.
        This gives $\mu_{(2), 0} = -1/2$. Note that $(\mathbf{u}_2, r_2) = ((2), 0) \neq ((1), 0)$, so the non-triviality condition of Theorem \ref{M2} is satisfied.
        \item Right part ($i=1$): $e_0^{S_1-1} W_{\mathbf{u}_1} = e_1$. 
Hence $S_1=1, \mathbf{u}_1=(1)$.
        This gives $\mu_{(1), 0} = -1$.
    \end{itemize}
    We have $\mathbf{S} = (S_1, S_2) = (1, 1)$. 
The refinement of $\mathbf{S}=(1,1)$ is $(1,1)$.
    Contribution:
    \[ (+1) \cdot \left( Li^{\ell}_{1,1}(z;\gamma,\sigma) \right) \cdot \left( \mu_{(2),0} \cdot \mu_{(1),0} \right) = Li^{\ell}_{1,1}(z;\gamma,\sigma) \cdot \left( -\frac{1}{2} \right) \cdot (-1) = \frac{1}{2} Li^{\ell}_{1,1}(z;\gamma,\sigma). \]
\end{itemize}
Combining these results, we obtain the following formula.
\begin{prop}[Explicit $\ell$-adic Landen formula for $\mathbf{k}=(1,2)$]\label{sceq12}
Let the notation and assumptions be as in Theorem \ref{M2}.
For any $\sigma \in G_K$, the following holds:
\begin{align}
    &Li^{\ell}_{1,2}\left(\dfrac{z}{z-1};\gamma',\sigma\right) - Li^{\ell}_{1,2}(z;\gamma,\sigma) - Li^{\ell}_{1,1,1}(z;\gamma,\sigma)\\
    & \quad \quad \quad = \frac{1}{12} Li^{\ell}_{1}(z;\gamma,\sigma) + \frac{1}{2} Li^{\ell}_{1,1}(z;\gamma,\sigma). \notag
\end{align}
\end{prop}

\subsection{Case $\mathbf{k}=(2,1)$ (Weight 3, Depth 2)}
Let $\mathbf{k}=(2,1)$. 
The associated monomial is
$$W_{\mathbf{k}} = e_0^{1-1}e_1 e_0^{2-1}e_1 = e_1 e_0 e_1.$$
\paragraph{\underline{Left-hand side}}
The sign factor is $(-1)^{1+\depth(\mathbf{k})}=-1$.
The refinements of $(2,1)$ are $(2,1)$ and $(1,1,1)$.
Thus, the left-hand side is:
$$Li^{\ell}_{2,1}\left(\dfrac{z}{z-1};\gamma',\sigma\right) - \left( Li^{\ell}_{2,1}(z;\gamma,\sigma) + Li^{\ell}_{1,1,1}(z;\gamma,\sigma) \right).$$
\paragraph{\underline{Right-hand side}}
We sum over partitions for $M=1$ and $M=2$.
\begin{itemize}\item \textbf{Case $M=1$:}
The global sign is $(-1)^{M+\depth(\mathbf{k})} = -1$.
Equation: $e_0^{S_1-1} W_{\mathbf{u}_1} = e_1 e_0 e_1$.
The only solution ending in $e_1$ is $S_1=1$ and $\mathbf{u}_1=(2,1)$.
Thus $\mathbf{S}=(1)$.
We compute $\mu_{(2,1), 0}$.
The index list for $\mathbf{u}=(2,1)$ is derived from $(1-1, 1, 2-1, 1, 0) = (0, 1, 1, 1, 0) \to (1,1,1)$, so $L=3$.
The integral is:
$$\int_0^1 t^{\lfloor 3/2 \rfloor} (t-1)^{\lfloor 2/2 \rfloor} G_1(t)^3 \, dt = \int_0^1 t(t-1) \, dt = -\frac{1}{6}.$$
Since $u_m=u_2=1$,  the sign $\sigma_{\mathbf{u}, r}$ in (\ref{GPint}) is $(-1)^{3}=-1$. 
Thus $\mu_{(2,1), 0} = 1/6$.
Contribution:
$$(-1) \cdot Li^{\ell}_1(z;\gamma,\sigma) \cdot \frac{1}{6} = -\frac{1}{6} Li^{\ell}_1(z;\gamma,\sigma).$$
\item \textbf{Case $M=2$:}
The global sign is $(-1)^{M+\depth(\mathbf{k})} = +1$.
We decompose $W_{(2,1)} = e_1 e_0 e_1$ into $(e_0^{S_2-1}W_{\mathbf{u}_2} e_0^{r_2}) \cdot (e_0^{S_1-1}W_{\mathbf{u}_1})$.
The split must occur around the middle $e_0$.
There are two ways to distribute this $e_0$:
\begin{enumerate}\item \textbf{Split after the middle $e_0$ (i.e., $r_2=1$):}
$$\underbrace{e_1 e_0}_{\text{Left}} \cdot \underbrace{e_1}_{\text{Right}}$$
\begin{itemize}
\item Right part: 
$e_0^{S_1-1} W_{\mathbf{u}_1} = e_1 \implies S_1=1, \mathbf{u}_1=(1)$.
Coefficient: $\mu_{(1), 0} = -1$.    
\item Left part: $e_0^{S_2-1} W_{\mathbf{u}_2} e_0^{r_2} = e_1 e_0$.
Here $r_2=1$. 
This implies $e_0^{S_2-1} W_{\mathbf{u}_2} = e_1$.
Thus $S_2=1, \mathbf{u}_2=(1)$.
    
We need to calculate $\mu_{(1), 1}$ where $r=1$.
The index list for $\mathbf{u}=(1), r=1$ is $(1-1, 1, 1) = (0, 1, 1) \to (1,1)$, so $L=2$.
$$\int_0^1 t^{\lfloor 2/2 \rfloor} (t-1)^{\lfloor 1/2 \rfloor} G_1(t) G_1(t) \, dt = \int_0^1 t \, dt = \frac{1}{2}.$$
Since $u_m=1$,  the sign $\sigma_{\mathbf{u}, r}$ in (\ref{GPint}) is $(-1)^{\mathrm{wt}(\mathbf{u})+r} = (-1)^{1+1} = +1$.
Thus, $\mu_{(1), 1} = 1/2$.
\end{itemize}
This combination gives $\mathbf{S}=(1,1)$.
The contribution is $$(+1) \cdot Li^{\ell}_{1,1}(z;\gamma,\sigma) \cdot (\mu_{(1),1} \mu_{(1),0}) = -\frac{1}{2} Li^{\ell}_{1,1}(z;\gamma,\sigma).$$

\item \textbf{Split before the middle $e_0$ (i.e., $r_2=0$):}
$$\underbrace{e_1}_{\text{Left}} \cdot \underbrace{e_0 e_1}_{\text{Right}}$$
\begin{itemize}
\item Left part: $e_0^{S_2-1} W_{\mathbf{u}_2} = e_1 \implies S_2=1, \mathbf{u}_2=(1)$.
So $(\mathbf{u}_2,r_2)=((1),0)$ with coefficient $\mu_{(1), 0} = -1$.
    
\item Right part: $e_0^{S_1-1} W_{\mathbf{u}_1} = e_0 e_1$. 
This yields two possibilities:
\begin{itemize}
\item $S_1=2, \mathbf{u}_1=(1)$. 
So $(\mathbf{u}_1,r_1)=((1),0)$. Both pairs are trivial, so this combination is excluded from the sum.
\item $S_1=1, \mathbf{u}_1=(2)$. So $(\mathbf{u}_1,r_1)=((2),0)$. By Proposition \ref{prop:integral_rep_corrected2}, the coefficient is $\mu_{(2), 0} = -1/2$.
\end{itemize}
\end{itemize}
For the non-excluded subcase ($S_1=1, \mathbf{u}_1=(2)$), this combination gives $\mathbf{S}=(1,1)$.
The contribution is $$(+1) \cdot Li^{\ell}_{1,1}(z;\gamma,\sigma) \cdot (\mu_{(1),0} \mu_{(2),0}) = \frac{1}{2} Li^{\ell}_{1,1}(z;\gamma,\sigma).$$
\end{enumerate}
Combining the two non-excluded combinations for $M=2$, the total contribution is   
$$-\frac{1}{2} Li^{\ell}_{1,1}(z;\gamma,\sigma) + \frac{1}{2} Li^{\ell}_{1,1}(z;\gamma,\sigma) = 0.$$
\end{itemize}
Combining these results, we obtain the following formula.
\begin{prop}[Explicit $\ell$-adic Landen formula for $\mathbf{k}=(2,1)$]\label{sceq21}
Let the notation and assumptions be as in Theorem \ref{M2}.
For any $\sigma \in G_K$, the following holds:
\begin{align}
Li^{\ell}_{2,1}\left(\dfrac{z}{z-1};\gamma',\sigma\right) - Li^{\ell}_{2,1}(z;\gamma,\sigma) - Li^{\ell}_{1,1,1}(z;\gamma,\sigma) = -\frac{1}{6} Li^{\ell}_1(z;\gamma,\sigma).
\end{align}
\end{prop}

\subsection{Case $\mathbf{k}=(1,1,2)$ (Weight 4, Depth 3)}
Let $\mathbf{k}=(1,1,2)$.
The associated word is
$$W_{\mathbf{k}} = e_0 e_1 e_1 e_1.$$
\paragraph{\underline{Left-hand side}}
The sign factor is $(-1)^{1+\depth(\mathbf{k})}=+1$.
The refinements of $\mathbf{k}=(1,1,2)$ are $(1,1,2)$ itself and $(1,1,1,1)$.
Thus, the left-hand side is given by:$$Li^{\ell}_{1,1,2}\left(\dfrac{z}{z-1};\gamma',\sigma\right) + Li^{\ell}_{1,1,2}(z;\gamma,\sigma) + Li^{\ell}_{1,1,1,1}(z;\gamma,\sigma).$$
\paragraph{\underline{Right-hand side}}
We analyze the right-hand side by the length $M$ of the partition of $W_{\mathbf{k}}$.
We compute the necessary coefficients $\mu_{\mathbf{u}, 0}$ using Proposition \ref{prop:integral_rep_corrected} as follows:
$$\mu_{(1), 0} = -1, \quad \mu_{(2), 0} = -\frac{1}{2}, \quad \mu_{(1,2), 0} = -\frac{1}{12},$$
$$\mu_{(1,1), 0} = 0, \quad \mu_{(1,1,1), 0} = 0, \quad \mu_{(1,1,2), 0} = 0.$$
Using these results, we compute the summation in the error terms.
\begin{enumerate}\item \textbf{Case $M=1$:} 
The global sign is $(-1)^{M+\depth(\mathbf{k})} = +1$.
We solve $e_0^{S_1-1} W_{\mathbf{u}_1} = e_0 e_1^3$. 
\begin{itemize}
\item If $S_1=1$, then $\mathbf{u}_1=(1,1,2)$.
Since $\mu_{(1,1,2), 0} = 0$, this term contributes $0$.\item If $S_1=2$, then 
$\mathbf{u}_1=(1,1,1)$.
Since $\mu_{(1,1,1), 0} = 0$, this term contributes $0$.\end{itemize}Thus, the contribution for $M=1$ is identically zero.
\item \textbf{Case $M=2$:}
The global sign is $(-1)^{M+\depth(\mathbf{k})} = -1$.
We decompose $e_0 e_1^3$ into two words $W_{\mathbf{u}_2} \cdot W_{\mathbf{u}_1}$. 
There are two possible splits:
\begin{itemize}
    \item Split as $(e_0 e_1) \cdot (e_1 e_1)$.
    Here $W_{\mathbf{u}_1} = e_1 e_1$. Since $\mu_{(1,1),0} = 0$, this term vanishes.
    
    \item Split as $(e_0 e_1 e_1) \cdot (e_1)$. 
This implies $\mathbf{S}=(1,1)$ (since $W_{\mathbf{u}_2} = e_0 e_1 e_1$ consumes one $e_0$, and the total requires one $e_0$).
    The coefficient is $\mu_{(1,2),0}\mu_{(1),0} = (-1/12)(-1) = 1/12$.
    The contribution is:
    \[
    (-1) \cdot \left( \frac{1}{12} \right) \cdot Li^{\ell}_{1,1}(z;\gamma,\sigma) = - \frac{1}{12} Li^{\ell}_{1,1}(z;\gamma,\sigma).
    \]
\end{itemize}

\item \textbf{Case $M=3$:}
The global sign is $(-1)^{M+\depth(\mathbf{k})} = +1$.
The only split is $(e_0 e_1) \cdot (e_1) \cdot (e_1)$, corresponding to $\mathbf{S}=(1,1,1)$.
The coefficient is $\mu_{(2),0}\mu_{(1),0}\mu_{(1),0} = (-1/2)(-1)(-1) = -1/2$.
The contribution is:
\[
(+1) \cdot \left(-\frac{1}{2}\right) \cdot Li^{\ell}_{1,1,1}(z;\gamma,\sigma) = -\frac{1}{2} Li^{\ell}_{1,1,1}(z;\gamma,\sigma).
\]
\end{enumerate}
Summing all non-vanishing contributions, we obtain the explicit formula.
\begin{prop}[Explicit $\ell$-adic Landen formula for $\mathbf{k}=(1,1,2)$]\label{sceq112}
Let the notation and assumptions be as in Theorem \ref{M2}.
For any $\sigma \in G_K$, the following holds:
\begin{align}Li^{\ell}_{1,1,2}\left(\dfrac{z}{z-1};\gamma',\sigma\right) &+ Li^{\ell}_{1,1,2}(z;\gamma,\sigma) + Li^{\ell}_{1,1,1,1}(z;\gamma,\sigma) \\ &= -\frac{1}{12} Li^{\ell}_{1,1}(z;\gamma,\sigma) - \frac{1}{2} Li^{\ell}_{1,1,1}(z;\gamma,\sigma). \notag
\end{align}\end{prop}

\subsection{Case $\mathbf{k}=(2,3)$ (Weight 5, Depth 2)}
Let $\mathbf{k}=(2,3)$. The associated monomial is
$$W_{\mathbf{k}} = e_0^{3-1}e_1 e_0^{2-1}e_1 = e_0^2 e_1 e_0 e_1.
$$\paragraph{\underline{Left-hand side}}
The global sign factor is $(-1)^{1+\depth(\mathbf{k})}=-1$.
The refinements $\mathbf{J} \succeq (2,3)$ are generated by combining the refinements of the component indices $2$ and $3$.
\begin{itemize}
\item Refinements of $(2)$: $(2),~(1,1)$.
\item Refinements of $(3)$: $(3),~(2,1),~(1,2),~(1,1,1)$.\end{itemize}
Taking the product, there are $2 \times 4 = 8$ refinements:
$$(2,3),~(2,2,1),~(2,1,2),~(2,1,1,1),~(1,1,3),~(1,1,2,1),~(1,1,1,2),~(1,1,1,1,1).$$
Thus, the left-hand side is:
\begin{align*}
Li^{\ell}_{2,3}\left(\dfrac{z}{z-1};\gamma',\sigma\right) - &\left( Li^{\ell}_{2,3}(z;\gamma,\sigma)+Li^{\ell}_{2,2,1}(z;\gamma,\sigma)+Li^{\ell}_{2,1,2}(z;\gamma,\sigma)+Li^{\ell}_{2,1,1,1}(z;\gamma,\sigma)\right.\\
&\left.+Li^{\ell}_{1,1,3}(z;\gamma,\sigma)+Li^{\ell}_{1,1,2,1}(z;\gamma,\sigma)+Li^{\ell}_{1,1,1,2}(z;\gamma,\sigma)+Li^{\ell}_{1,1,1,1,1}(z;\gamma,\sigma)\right).
\end{align*}
\paragraph{\underline{Right-hand side}}
The relevant cases are $M=1$ and $M=2$.
\begin{itemize}
\item \textbf{Case $M=1$:} 
The global sign is $(-1)^{M+\depth(\mathbf{k})}=-1$.
We solve $e_0^{S_1-1} W_{\mathbf{u}_1} = e_0^2 e_1 e_0 e_1$ ($r_1=0$).
\begin{enumerate}
\item If $S_1=1$, then $W_{\mathbf{u}_1} = e_0^2 e_1 e_0 e_1$, i.e., $\mathbf{u}_1=(2,3)$.
The coefficient is $\mu_{(2,3), 0} = 1/120$.
The contribution is $-\frac{1}{120} Li^{\ell}_1$.

\item If $S_1=2$, then $W_{\mathbf{u}_1} = e_0 e_1 e_0 e_1$, i.e., $\mathbf{u}_1=(2,2)$.
The coefficient is $\mu_{(2,2), 0} = 1/12$.
The contribution is $-\frac{1}{12} (Li^{\ell}_2 + Li^{\ell}_{1,1})$.

\item If $S_1=3$, then $W_{\mathbf{u}_1} = e_1 e_0 e_1$, i.e., $\mathbf{u}_1=(2,1)$.
The coefficient is $\mu_{(2,1), 0} = 1/6$.
The contribution is $-\frac{1}{6} (Li^{\ell}_3 + Li^{\ell}_{2,1} + Li^{\ell}_{1,2} + Li^{\ell}_{1,1,1})$.
\end{enumerate}

\item \textbf{Case $M=2$:} 
The sign is $(-1)^{M+\depth(\mathbf{k})}=+1$.
We decompose $W_{(2,3)}$ into $(e_0^{S_2-1}W_{\mathbf{u}_2}e_0^{r_2}) \cdot (e_0^{S_1-1}W_{\mathbf{u}_1}e_0^{r_1})$.
There are two splitting points. 
Note that we label $\mathbf{S}=(S_1, S_2)$.

\noindent
\textbf{Subcase A: Split after the first $e_1$ ($r_2=0$)}
$$\underbrace{e_0^2 e_1}_{\text{Left}} \cdot \underbrace{e_0 e_1}_{\text{Right}}$$
For the Right part ($e_0 e_1$), possible $(S_1, \mathbf{u}_1, r_1)$ are $(1, (2), 0)$ and $(2, (1), 0)$.
For the Left part ($e_0^2 e_1$), possible $(S_2, \mathbf{u}_2)$ are $(1, (3))$, $(2, (2))$, and $(3, (1))$.
Using the integral values $\mu_{(1),0}=-1,~\mu_{(2),0}=-1/2,~\mu_{(3),0}=-1/12$, we compute the coefficients $\mu_{\mathbf{u}_2,0}\mu_{\mathbf{u}_1,0}$:
\begin{itemize}
\item If $\mathbf{S}=(1,1)$, then $\mathbf{u}=((2),(3))$.
The coefficient is $(-1/2)(-1/12) = 1/24$.
Thus,
the contribution is $\frac{1}{24} Li^{\ell}_{1,1}$.

\item If $\mathbf{S}=(1,2)$, then $\mathbf{u}=((2),(2))$.
The coefficient is $(-1/2)(-1/2) = 1/4$.
Thus,
the contribution is $\frac{1}{4} (Li^{\ell}_{1,2} + Li^{\ell}_{1,1,1})$.

\item If $\mathbf{S}=(1,3)$, then $\mathbf{u}=((2),(1))$.
The coefficient is $(-1/2)(-1) = 1/2$.
Thus,
the contribution is $\frac{1}{2} (Li^{\ell}_{1,3}+Li^{\ell}_{1,2,1}+Li^{\ell}_{1,1,2}+Li^{\ell}_{1,1,1,1})$.

\item If $\mathbf{S}=(2,1)$, then $\mathbf{u}=((1),(3))$.
The coefficient is $(-1)(-1/12) = 1/12$.
Thus,
the contribution is $\frac{1}{12} (Li^{\ell}_{2,1} + Li^{\ell}_{1,1,1})$.

\item If $\mathbf{S}=(2,2)$, then $\mathbf{u}=((1),(2))$.
The coefficient is $(-1)(-1/2) = 1/2$.
Thus,
the contribution is $\frac{1}{2} (Li^{\ell}_{2,2}+Li^{\ell}_{1,1,2}+Li^{\ell}_{2,1,1}+Li^{\ell}_{1,1,1,1})$.

\item If $\mathbf{S}=(2,3)$, then $\mathbf{u}=((1),(1))$.
    The pairs $(\mathbf{u}_1,r_1)$ and $(\mathbf{u}_2,r_2)$ are trivial ((1),0). 
    Therefore, this case is excluded from the sum.
\end{itemize}\textbf{Subcase B: Split before the second $e_1$ ($r_2=1$)}
$$\underbrace{e_0^2 e_1 e_0}_{\text{Left}} \cdot \underbrace{e_1}_{\text{Right}}$$
Right part is fixed: $S_1=1, \mathbf{u}_1=(1), \mu_{(1),0}=-1$.
Left part involves $\mu_{\mathbf{u}_2, 1}$.
\begin{enumerate}
\item If $\mathbf{S}=(1,1)$, then $S_2=1, {\mathbf{u}_2}=(3)$. 
Since $\mu_{(3),1}=0$,
the contribution vanishes.

\item If $\mathbf{S}=(1,2)$, then $S_2=2, {\mathbf{u}_2}=(2)$.
$\mu_{(2),1}=1/6$.
The coefficient is $(-1)(1/6) = -1/6$.
The contribution is $-\frac{1}{6} (Li^{\ell}_{1,2} + Li^{\ell}_{1,1,1})$.

\item If $\mathbf{S}=(1,3)$, then $S_2=3, {\mathbf{u}_2}=(1)$.
$\mu_{(1),1}=1/2$.
The coefficient is $(-1)(1/2) = -1/2$.
The contribution is $-\frac{1}{2} (Li^{\ell}_{1,3} + Li^{\ell}_{1,2,1} + Li^{\ell}_{1,1,2} + Li^{\ell}_{1,1,1,1})$.
\end{enumerate}
\end{itemize}
Summing all non-vanishing contributions, we obtain the explicit formula.
\begin{prop}[Explicit $\ell$-adic Landen formula for $\mathbf{k}=(2,3)$]\label{sceq23}
Let the notation and assumptions be as in Theorem \ref{M2}. For any $\sigma \in G_K$, the following holds:
\begin{align}
&Li^{\ell}_{2,3}\left(\dfrac{z}{z-1};\gamma',\sigma\right) - \left( Li^{\ell}_{2,3}(z;\gamma,\sigma)+Li^{\ell}_{2,2,1}(z;\gamma,\sigma)+Li^{\ell}_{2,1,2}(z;\gamma,\sigma)+Li^{\ell}_{2,1,1,1}(z;\gamma,\sigma)\right.\\
&\quad \quad \left.+Li^{\ell}_{1,1,3}(z;\gamma,\sigma)+Li^{\ell}_{1,1,2,1}(z;\gamma,\sigma)+Li^{\ell}_{1,1,1,2}(z;\gamma,\sigma)+Li^{\ell}_{1,1,1,1,1}(z;\gamma,\sigma)\right)\notag \\
&\quad = -\frac{1}{120} Li^{\ell}_1(z;\gamma,\sigma) - \frac{1}{12} Li^{\ell}_2(z;\gamma,\sigma) - \frac{1}{24} Li^{\ell}_{1,1}(z;\gamma,\sigma) \notag \\
&\quad \quad - \frac{1}{6} Li^{\ell}_3(z;\gamma,\sigma) - \frac{1}{12} Li^{\ell}_{2,1}(z;\gamma,\sigma) - \frac{1}{12} Li^{\ell}_{1,2}(z;\gamma,\sigma) \notag \\
&\quad \quad + \frac{1}{2} \left( Li^{\ell}_{2,2}(z;\gamma,\sigma) + Li^{\ell}_{2,1,1}(z;\gamma,\sigma) + Li^{\ell}_{1,1,2}(z;\gamma,\sigma) + Li^{\ell}_{1,1,1,1}(z;\gamma,\sigma) \right).\notag
\end{align}
\end{prop}

\section*{Acknowledgements}\noindent
This paper was written during the author's visit to ``Institut de Mathématiques de Jussieu-Paris Rive Gauche'' under the Autumn in Paris 2025 exchange program, which is part of the ``Arithmetic and Homotopic Galois Theory'' CNRS-RIMS International Research Network.
The author is grateful to Benjamin Collas for support during the stay and to Ariane Mézard for providing a working environment.
The author also thanks Tokyo University of Science and colleagues there for providing financial support for the stay and for their assistance.
The author would like to thank the anonymous referee for the careful reading of the manuscript and for many valuable comments and suggestions, which have substantially improved the paper.
This work was supported by JSPS KAKENHI Grant Numbers JP24K22840 and JP26K16966.

\end{document}